\documentclass[leqno]{siamltex}

\usepackage{amsmath} 
\usepackage{amssymb}  
\usepackage{enumerate}

\usepackage{graphicx}
\usepackage{graphics,subfigure}

\newtheorem{example}{Example}
\newtheorem{algorithm}{Algorithm}
\newtheorem{experiment}{Experiment}

\title{Generalized Network Tomography \thanks{A summary of this work without proofs has appeared in the proceedings of the 50th Annual Allerton Conference on Communication, Control and Computing (see \cite{GuganC12}).}}

\author{Gugan Thoppe\thanks{School of Technology \& Computer Science, Tata Institute of Fundamental Research, 1st Homi Bhabha Road, Mumbai, INDIA ({\tt gugan@tcs.tifr.res.in}).}}

\begin{document}

\maketitle

\begin{abstract}
Generalized network tomography (GNT) deals with estimation of link performance parameters for networks with arbitrary topologies using only end-to-end path  measurements of pure unicast probe packets. In this paper, by taking advantage of the properties of generalized hyperexponential distributions and polynomial systems, a novel algorithm to infer the complete link metric distributions under the framework of GNT is developed. The significant advantages of this algorithm are that it does not require: i) the path measurements to be synchronous and ii) any prior knowledge of the link metric distributions. Moreover, if the path-link matrix of the network has the property that every pair of its columns are linearly independent, then it is shown that the algorithm can uniquely identify the link metric distributions up to any desired accuracy. Matlab based simulations have been included to illustrate the potential of the proposed scheme.
\end{abstract}



\pagestyle{myheadings}
\thispagestyle{plain}
\markboth{GUGAN THOPPE}{GENERALIZED NETWORK TOMOGRAPHY}

\section{Introduction}
The present age Internet is a massive, heterogeneous network of networks with a decentralized control. Despite this, accurate, timely and localized information about its connectivity, bandwidth and performance measures such as average delay experienced by traffic, packet loss rates across links, etc. is extremely vital for its efficient management. Brute force techniques, such as gathering the requisite information directly, impose an impractical overhead and hence are generally avoided. This necessitated the advent of network tomography---the science of inferring spatially localized network behaviour using only end-to-end aggregate metrics.

Recent advances in network tomography can be classified into two broad strands: i) traffic demand tomography---determination of source-destination traffic volumes via measurements of link volumes and ii) network delay tomography---link parameter estimation based on end-to-end path level measurements. For the first strand, see \cite{Vardi96, LeBlanc82, Zhang03a}. Under the second strand, the major problems studied include estimation of bottleneck link bandwidths, e.g. \cite{Liu07, Dey11}, link loss rates, e.g. \cite{Caceres99}, link delays, e.g. \cite{Coates01, LoPresti02, Shih03, Tsang03, Chen10, Deng12}, etc. Apart from these, there is also work on estimation of the topology of the network via path measurements. For excellent tutorials and surveys on the state of the art, see \cite{Adams00, Coates01, Castro04, Lawrence07}. For sake of definiteness, we consider here the problem of network delay tomography. The proposed solution is, however, also applicable to traffic demand tomography.

Given a binary matrix $A,$ usually called the path-link matrix, the central problem in network delay tomography, in abstract terms, is to accurately estimate the statistics of the vector $X$ from the measurement model $Y = AX.$ Based on this, existing work can be categorized into deterministic and stochastic approaches. Deterministic approaches, e.g. \cite{Gurewitz01, Chen04, Firooz10}, treat $X$ as a fixed but unknown vector and use linear algebraic techniques to solve for $X.$ Clearly, when no prior knowledge is available, $X$ can be uniquely recovered only when $A$ is invertible, a condition often violated in practice. Stochastic approaches, e.g. \cite{Chen10, Shih03, Tsang03, Zhang03a}, on the other hand, assume $X$ to be a non-negative random vector of mutually independent components and employ parametric/non-parametric estimation techniques to infer the statistical properties of $X$ using samples of $Y.$ In this paper, we build a stochastic network tomography scheme and establish sufficient conditions on $A$ for accurate identification of the distribution of $X.$

Stochastic network tomography approaches, in general, model the distribution of each component of $X$ using either a discrete distribution, e.g. \cite{Tsang03, Zhang03a}, or a finite mixture model, e.g. \cite{Chen10, Shih03}. They construct an optimization problem based on the characteristic function, e.g. \cite{Chen10}, or a suitably chosen likelihood function, e.g. \cite{Shih03, Tsang03, Zhang03a}, of $Y.$ Algorithms such as expectation-maximization, e.g. \cite{Shih03, Tsang03, Zhang03a}, generalized method of moments, e.g. \cite{Chen10}, etc., which mainly exploit the correlations in the components of $Y,$ are then employed to determine the optimal statistical estimates of $X.$ In practice, however, these algorithms suffer two main limitations. Firstly, note that these algorithms utilize directly the samples of the vector $Y.$ Thus, to implement them, one would crucially require i) end-to-end data generated using multicast probe packets, real or emulated, and ii) the network to be a tree rooted at a single sender with destinations at leaves. Divergence in either of the above requirements, which is often the case, thus results in performance degradation. Secondly, the optimization problems considered tend to have multiple local optima. Thus, without prior knowledge, the quality of the estimate is difficult to ascertain. 

In this paper, we consider the problem of generalized network tomography (GNT) wherein, the objective is to estimate the link performance parameters for networks with arbitrary topologies using only end-to-end measurements of pure unicast probe packets. Mathematically, given a binary matrix $A,$ we propose a novel method, henceforth called the distribution tomography (DT) scheme, to accurately estimate the distribution of $X,$ a vector of independent non-negative random variables, using only IID samples of the components of the random vector $Y = AX.$ In fact, our scheme does not even require prior knowledge of the distribution of $X.$ We thus overcome the limitations of the previous approaches. 

We rely on the fact that the class of \textit{generalized hyperexponential} (GH) distributions is dense in the set of non-negative distributions (see \cite{Botha86}). Using this, the idea is to approximate the distribution of each component of $X$ using linear combinations of known exponential bases and estimate the unknown weights. These weights are obtained by solving a set of polynomial systems based on the moment generating function of the components of  $Y.$ For unique identifiability, it is only required that every pair of columns of the matrix $A$ be linearly independent, a property that holds true for the path-link matrix of all multicast tree networks and more.

The rest of the paper is organized as follows. In the next section, we develop the notation and formally describe the problem. Section~\ref{sec:approxDist} recaps the theory of approximating non-negative distributions using linear combinations of exponentials. In Sections~\ref{sec:DistTomoSc} and \ref{sec:universality}, we develop our proposed method and demonstrate its universal applicability. We give numerical examples in Section~\ref{sec:ExptResults} and end with a short discussion in Section~\ref{sec:Disc}.

We highlight at the outset that the aim of this paper is to establish the theoeretical justification for the proposed scheme. The numerical examples presented are only for illustrative purposes. 

\section{Model and Problem Description}
\label{sec:model}
Any cumulative distribution function (CDF) that we work with is always assumed to be continuous with support $(0, \infty).$ The moment generating function (MGF) of the random variable $X$ will be $M_X(t) = \mathbb{E}(\exp(-tX)).$ For $n \in \mathbb{Z}_{++},$ we use $[n]$ and $S_n$ to represent respectively the set $\{1, \ldots, n\}$ and its permutation group. We use $\tbinom{n}{k_1, k_2, \cdots, k_d}$ to represent $\frac{n!}{k_1! \cdots k_d!}.$ Similarly, $\tbinom{n}{k}$ stands for $\frac{n!}{k! (n - k)!}.$ We use the notation $\mathbb{R}, \mathbb{R}_{+}$ and $\mathbb{R}_{++}$ to denote respectively the set of real numbers, non-negative real numbers and strictly positive real numbers. In the same spirit, for integers, we use $\mathbb{Z}, \mathbb{Z}_{+}$ and $\mathbb{Z}_{++}.$ All vectors are column vectors and their lengths refer to the usual Euclidean norm. For $\delta > 0,$ $B(\mathbf{v}; \delta)$ represents the open $\delta-$ball around the vector $\mathbf{v}.$ To denote the derivative of the map $f$ with respect to $\mathbf{x},$ we use $\dot{f}(\mathbf{x}).$ Lastly, all empty sums and empty products equal $0$ and $1$ respectively.

Let $X_1, \ldots, X_N$ denote the independent non-negative random variables whose distribution we wish to estimate. We assume that each $X_j$ has a GH distribution of the form
\begin{equation}
  \begin{array}{cccc}
    F_j(u) & = & \sum\limits_{k=1}^{d + 1} w_{jk} \left[ 1 - \exp\left( -\lambda_{k} u \right) \right],  & u \geq 0
  \end{array}
	\label{eq:Dist}
\end{equation}
where $\lambda_{k}>0,$ $\sum_{k=1}^{d+1} w_{jk} \lambda_{k} \exp(-\lambda_{k}u) \geq 0$ and $\sum_{k=1}^{d + 1}w_{jk} =  1.$ Note that the weights $\{w_{jk}\}$, unlike for hyperexponential distributions, are not required to be all positive. Further, we suppose that $\lambda_1, \ldots, \lambda_{d + 1}$ are distinct and explicitly known and that the weight vectors of distinct random variables differ at least in one component. Let $A \in \{0,1\}^{m \times N}$ denote an a priori known matrix which is $1-$identifiable in the following sense.
\begin{definition}
A matrix $A$ is $k-$identifiable if every set of $2k$  of its columns is linearly independent.
\end{definition}

Let $X \equiv (X_1, \ldots, X_N)$ and $Y = AX.$ For each $i \in [m],$ let $p_i := \{j \in [N]: a_{ij} = 1\}.$ Further, we  presume that, $\forall i \in [m],$  we have access to a sequence of IID samples $\{Y_{il}\}_{l \geq 1}$ of $Y_i.$ Our problem then is to estimate for each $X_j,$ its vector of weights $\mathbf{w}_j \equiv (w_{j1}, \ldots, w_{jd})$ and consequently its complete distribution $F_j,$ since $w_{j(d+1)} = 1 - \sum_{k = 1}^{d} w_{jk}.$

Before developing the estimation procedure, we begin by making a case for the distribution model of (\ref{eq:Dist}).
 
\section{Approximating distribution functions} 
\label{sec:approxDist}
Let $\mathcal{G} = \{G^1, \ldots, G^N\}$ denote a finite family of arbitrary non-negative distributions. For the problem of simultaneously estimating all member of $\mathcal{G}$ a useful strategy, as we demonstrate now, is to approximate each $G^i$ by a GH distribution.

Recall that the CDF of a GH random variable $X$ is given by
\begin{equation} \label{eq:GHCDF}
F_X(u) = \sum_{k = 1}^{d + 1}\alpha_k [1 - \exp(-\lambda_k u)], \hspace{.2in} u \geq 0,
\end{equation}
where $\lambda_k > 0,$ $\sum_{k = 1}^{d + 1}\alpha_k \lambda_k \exp(-\lambda_k u) \geq 0$ and $\sum_{k = 1}^{d + 1} \alpha_k = 1.$ Consequently, its MGF is given by
\begin{equation}
M_X(t) = \sum_{k = 1}^{d + 1}\alpha_k\frac{\lambda_k}{\lambda_k + t}.
\end{equation}
In addition to the simple algebraic form of the above quantities, the other major reason to use the GH class is that, in the sense of weak topology, it is dense in the set of all non-negative distributions (see \cite{Botha86}). In fact, as the following result from \cite{Ou97} shows, we have much more.

\begin{theorem} \label{thm:Ou}
For $n,k \in \mathbb{Z}_{++},$ let $X_{n,k}$ be a nonnegative GH random variable with mean $k/n,$ variance $\sigma_{n,k}^{2}$ and CDF $W_{n,k}.$ Suppose
  \begin{enumerate}
  \item the function $\nu:\mathbb{Z}_{++}\rightarrow\mathbb{Z}_{++}$ satisfies $\underset{n\rightarrow\infty}{lim}\nu(n)/n=\infty.$
  \item there exists $0<s<1$ such that $\underset{n\rightarrow\infty}{lim}n^{1+s} \sigma_{n,k}^{2}/k=0$ uniformly with respect to $k.$
  \end{enumerate}
  Then given any continuous non-negative distribution function $F,$ the following holds:
  \begin{enumerate}
  \item the function $F_{n}$ given by
    \begin{eqnarray*}
      F_{n}(u) & = & \sum_{k=1}^{\nu(n)}\left\{F(k/n)-F((k-1)/n)\right\} W_{n,k}(u)\\
      &  & +(1-F(\nu(n)/n))W_{n,\nu(n)+1}(u)
    \end{eqnarray*}
    is a GH distribution for every $n \in \mathbb{Z}_{++}$ and
  \item $F_{n}$ converges uniformly to $F,$ i.e., 
  \begin{equation*}
    \underset{n \rightarrow \infty}{\lim}\; \underset{-\infty < u < \infty}{\sup} \left|F_{n}(u)-F(u)\right| =  0.
  \end{equation*}\,
  \end{enumerate}
\end{theorem}

Observe that, for each $n,$ the exponential stage parameters of $F_n$ depend only on the choice of the random variables $\{X_{n,k}\!: 1\leq k\leq \nu(n) +1\}.$

What this observation and the above result imply in relation to $\mathcal{G}$ is that if we fix the random variables $\{X_{n,k}\}$ and let $M^i$ denote the MGF of  $G^i,$  then for any given $\epsilon_1, \epsilon_2 >0$ and any finite set $\tau=\{t_1,\ldots ,t_k\} \subset \mathbb{R}_+,$ $\exists n \equiv n(\epsilon_1, \epsilon_2, \tau) \in \mathbb{Z}_{++}$ such that for each $i \in [N],$ $G^i$ and its $n^{th}$ GH approximation, $F_n^i,$ are $\epsilon_1-$ close in the sup norm and for each $j \in [k],$ $|M^i(t_j)-M_n^i(t_j)|\leq \epsilon_2.$ Further, the exponential stage parameters are explicitly known and identical across the approximations $F_n^1, \ldots, F_n^N,$ which now justifies our model of (\ref{eq:Dist}).

The problem of estimating the individual members of $\mathcal{G}$ can thus be reduced to determining the vector of weights that characterizes each approximation and hence each distribution.

\section{Distribution tomography scheme} 
\label{sec:DistTomoSc}
The outline for this section is as follows. For each $i \in [m],$ we use the IID samples of $Y_i$ to estimate its MGF and subsequently build a polynomial system, say $H_i(\mathbf{x}) = 0.$ We call this the elementary polynomial system (EPS). We then show that for each $i \in [m]$ and each $j \in p_i,$ a close approximation of the vector $\mathbf{w}_j$ is present in the solution set of $H_i(\mathbf{x}) = 0,$ denoted $V(H_i).$ To match the weight vectors to the corresponding random variables, we make use of the fact that $A$ is $1-$identifiable.

\subsection{Construction of elementary polynomial systems}
\label{subsec:ConstPolySys}
Fix an arbitrary $i \in [m]$ and suppose that $|p_i| = N_i,$ i.e., $Y_i$ is a sum of $N_i$ random variables, which for notational convenience, we relabel as $X_1, \ldots, X_{N_i}.$ Because of independence of the random variables, observe that the MGF of $Y_i$ is well defined $\forall t \in \mathbb{R}_{++}$ and satisfies the relation
\begin{equation}
\label{eq:MGFRel}
M_{Y_{i}}(t) = \prod_{j = 1}^{N_i} \left\{ \sum_{k = 1}^{d + 1}w_{jk} \left( \frac{\lambda_{k}}{\lambda_{k}+t} \right) \right\}.
\end{equation}
On simplification, after substituting $w_{j(d+1)} = 1 - \sum_{k=1}^{d}w_{jk},$ we get
\begin{equation}
\label{eq:modMGFRel}
\mu_i(t) = \left\{ \prod_{j = 1}^{N_i} \left[ \sum_{k=1}^{d}w_{jk}\Lambda_k(t) + \lambda_{d + 1}\right] \right\},
\end{equation}
where $\Lambda_k(t) = (\lambda_{k} - \lambda_{d + 1})t/(\lambda_{k} + t)$ and $\mu_i(t) = M_{Y_{i}}(t) (\lambda_{d + 1} + t)^{N_i} .$

For now, let us assume that we know $M_{Y_i}(t)$ and hence $\mu_i(t)$ exactly for every valid $t.$ We will refer henceforth to this situation as the ideal case. Treating $t$ as a parameter, we can then use \eqref{eq:modMGFRel} to define a canonical polynomial
\begin{equation} \label{eq:tempGenPoly}
f(\mathbf{x};t) = \left\{ \prod_{j = 1}^{N_i} \left[ \sum_{k=1}^{d}x_{jk}\Lambda_k(t) + \lambda_{d + 1}\right] \right\} - \mu_i(t), 
\end{equation}
where $\mathbf{x} \equiv (\mathbf{x}_1, \ldots, \mathbf{x}_{N_i})$  with $\mathbf{x}_j \equiv (x_{j1}, \ldots, x_{jd}).$ As this is a multivariate map in $d \cdot N_i$ variables, we can choose an arbitrary set $\tau = \{t_1, \ldots, t_{d \cdot N_i}\} \subset \mathbb{R}_{++}$ consisting of distinct numbers and define an intermediate square polynomial system
\begin{equation}
\label{eq:tempPolSys}
F_\tau(\mathbf{x}) \equiv (f_1(\mathbf{x}), \ldots, f_{d \cdot N_i}(\mathbf{x})) = 0,
\end{equation}
where $f_k(\mathbf{x}) \equiv f(\mathbf{x}; t_k).$

Since \eqref{eq:tempPolSys} depends on choice of $\tau,$ analyzing it directly is difficult. But observe that i) the expansion of each $f_n$ or equivalently \eqref{eq:tempGenPoly} results in rational coefficients in $t$ of the form $\Lambda_{1}^{\omega_{1}}(t)\cdots\Lambda_{d}^{\omega_{d}}(t)\lambda_{d+1}^{\omega_{d+1}},$ where, for each $k,$ $\omega_k \in \mathbb{Z}_{+}$ and $\sum_{k=1}^{d+1}\omega_k = N_{i},$ and ii) the monomials that constitute each $f_n$ are identical. This suggests that one may be able to get a simpler representation for \eqref{eq:tempPolSys}. We do so in the following three steps, where the first two focus on simplifying \eqref{eq:tempGenPoly}.\vspace{0.1in}

\noindent \emph{Step1-Gather terms with common coefficients}: Let ${\boldsymbol \omega}$ denote the $d-$dimensional vector $(\omega_1, \ldots, \omega_d)$ and let $\boldsymbol \Omega \equiv (\boldsymbol \omega, \omega_{d + 1}).$ Also, let
\[
\Delta_{d+1,N_{i}}:=\left\{ \boldsymbol \Omega \in \mathbb{Z}_{+}^{d + 1}: \sum_{k=1}^{d+1}\omega_{k}=N_{i}\right\}.
\]
For a vector $\mathbf{b}\equiv(b_{1},\ldots,b_{N_{i}})\in[d+1]^{N_{1}},$ let its type be denoted by $\Theta(\mathbf{b})\equiv(\theta_{1}(\mathbf{b}),\ldots,\theta_{d+1}(\mathbf{b})),$ where $\theta_{k}(\mathbf{b})$ is the count of the element $k$ in $\mathbf{b}.$ For every $\boldsymbol \Omega \in\Delta_{d+1,N_{i}},$ additionally define the set
\[
\mathcal{B}_{\boldsymbol \Omega}=\left\{ \mathbf{b}\equiv(b_{1},\ldots,b_{N_{i}})\in[d+1]^{N_{i}}:\,\Theta(\mathbf{b}) = \boldsymbol \Omega\right\}
\]
and the polynomial $g(\mathbf{x}; \boldsymbol \Omega) = \sum_{\mathbf{b} \in\mathcal{B}_{\boldsymbol \Omega}}(\prod_{j=1,\, b_{j} \neq d+1}^{N_{i}}x_{jb_{j}}).$ For any $\boldsymbol \omega \in \mathbb{Z}_{+}^d,$ let $\Lambda^\mathbf{\boldsymbol \omega} \equiv \Lambda_{1}^{\omega_{1}}(t) \cdots \Lambda_{d}^{\omega_{d}}(t).$ Then collecting terms with common coefficients in \eqref{eq:tempGenPoly}, the above notations help us rewrite it as
\begin{equation} \label{eq:afterColComTerms}
f(\mathbf{x};t)=\sum_{\boldsymbol \Omega \in \Delta_{d+1,N_{i}}} g(\mathbf{x};\boldsymbol \Omega)\Lambda^{\boldsymbol \omega} \lambda_{d + 1}^{\omega_{d + 1}} - \mu_{i}(t).
\end{equation}

\noindent \emph{Step2-Coefficient expansion and regrouping}: Using an idea similar to partial fraction expansion for rational functions in $t$, the goal here is to decompose each $\Lambda^{\boldsymbol \omega}$ into simpler terms. For each $j,$ $k\in[d],$ let
\[
\beta_{jk}:=
\begin{cases}
\frac{\lambda_{j}\left(\lambda_{k}-\lambda_{d+1}\right)}{\left(\lambda_{j}-\lambda_{k}\right)}, & j \neq k,\\
1, & j = k.
\end{cases}
\]
For each $\boldsymbol \omega \in \mathbb{Z}_{+}^{d},$ let $\mathcal{D}(\boldsymbol \omega):= \{k \in [d]: \omega_{k} > 0\}.$ Further, if $\mathcal{D}(\boldsymbol \omega) \neq \emptyset,$ then $\forall k \in \mathcal{D}(\boldsymbol \omega)$ and $\forall q \in [\omega_k],$ let
\[
\bar{\Delta}_{kq}(\boldsymbol \omega) := \{\mathbf{s} \equiv (s_{1},\ldots,s_{d})\in\mathbb{Z}_{+}^{d}: s_{r} = 0, \, \forall r \in \mathcal{D}(\boldsymbol \omega)^{c} \cup \{k\}; \sum_{n=1}^{d}s_n = \omega_{k}-q\} \mbox{ and}
\]
\begin{equation} \label{eq:GammakqDefn}
\gamma_{kq}(\boldsymbol \omega):=\prod_{r \in \mathcal{D}(\boldsymbol \omega)}\beta_{kr}^{\omega_{r}}\left\{\sum_{\mathbf{s} \in \bar{\Delta}_{kq}(\boldsymbol \omega)}\prod_{r \in \mathcal{D}(\boldsymbol \omega)}\tbinom{\omega_{r} + s_{r} - 1}{\omega_{r} - 1}\beta_{rk}^{s_{r}}\right\}.
\end{equation}
The desired decomposition is now given in the following result.
\begin{lemma} \label{lem:coefExp}
If $\mathcal{D}(\boldsymbol \omega) \neq \emptyset,$ then
\begin{equation} \label{eq:coefExp}
\Lambda^{\boldsymbol \omega} = \sum_{k \in \mathcal{D}(\boldsymbol \omega)}\sum_{q = 1}^{\omega_{k}} \gamma_{kq}(\boldsymbol \omega)\Lambda_{k}^{q}(t)
\end{equation}
for all $t.$ Further, this expansion is unique.
\end{lemma}
\begin{proof}
See Appendix~\ref{sec: Apdx CTexp}. \qquad
\end{proof}

By applying this expansion to each coefficient in \eqref{eq:afterColComTerms} and regrouping, we have
\begin{equation} \label{eq:genCanPol}
f(\mathbf{x};t)=\sum_{k=1}^{d}\sum_{q=1}^{N_{i}}h_{kq}(\mathbf{x})\Lambda_{k}^{q}(t)\lambda_{d+1}^{N_{i}-q}-c(t),
\end{equation}
where $c(t)=\mu_i(t) - \lambda_{d+1}^{N_{i}}$ and
\begin{equation} \label{eq:elemPolWeiCombo}
h_{kq}(\mathbf{x})=\sum_{\boldsymbol \Omega \in \Delta_{d+1,N_{i}},\, \omega_{k} \geq q}\frac{\gamma_{kq}(\boldsymbol \omega)}{\lambda_{d+1}^{N_{i} - q - \omega_{d+1}}}g(\mathbf{x};\boldsymbol \Omega). 
\end{equation}

We consider a simple example to better illustrate the above two steps. 

\begin{example} \label{eg:sEg}
Suppose $N_i = 3$ and $d = 2.$ In this case, clearly $\Lambda_{k}(t)=\frac{\left(\lambda_{k}-\lambda_{3}\right)t}{\lambda_{k}+t},$ $\mu_{i}(t)=(\lambda_{3}+t)^{3}M_{Y_{i}}(t)$ and $\mathbf{x}\equiv(x_{11},x_{12},x_{21},x_{22},x_{31},x_{32}).$ Equation (\ref{eq:tempGenPoly}) is 
\[
f(\mathbf{x};t) = \prod_{j = 1}^{3} \left[x_{j1}\Lambda_{1}(t)+x_{j2}\Lambda_{2}(t)+\lambda_{3}\right] - \mu_i(t),
\]
while \eqref{eq:afterColComTerms} is
\begin{eqnarray*}
f(\mathbf{x};t) & = & x_{11}x_{21}x_{31}\Lambda_{1}^{3}(t)+x_{12}x_{22}x_{32}\Lambda_{2}^{3}(t) \hspace{0.05in} + \\
 &  & (x_{11}x_{21}+x_{11}x_{31}+x_{21}x_{31})\Lambda_{1}^{2}(t)\lambda_{3} \hspace{0.05in} + \\
 &  & (x_{12}x_{22}+x_{12}x_{32}+x_{22}x_{32})\Lambda_{2}^{2}(t)\lambda_{3} \hspace{0.05in} +\\
 &  & \left(x_{11}x_{21}x_{32}+x_{11}x_{22}x_{31}+x_{12}x_{21}x_{31}\right)\Lambda_{1}^{2}(t)\Lambda_{2}(t) \hspace{0.05in}+\\
 &  & \left(x_{11}x_{22}x_{32}+x_{12}x_{21}x_{32}+x_{12}x_{22}x_{31}\right)\Lambda_{1}(t)\Lambda_{2}^{2}(t) \hspace{0.05in}+\\
 &  & \left(x_{11}+x_{21}+x_{31}\right)\Lambda_{1}(t)\lambda_{3}^{2}+\left(x_{12}+x_{22}+x_{32}\right)\Lambda_{2}(t)\lambda_{3}^{2} \hspace{0.05in} \hspace{0.05in}+\\
 &  & \left(x_{11}x_{22}+x_{11}x_{32}+x_{12}x_{21}+x_{21}x_{32}+x_{12}x_{31}+x_{22}x_{31}\right)\Lambda_{1}(t)\Lambda_{2}(t)\lambda_{3}\\
 &  & + \hspace{0.05in} \lambda_{3}^{3}-\mu_{i}(t).
\end{eqnarray*}
Now observe that 
\begin{equation*}
\Lambda_{1}(t)\Lambda_{2}(t) = \beta_{12}\Lambda_{1}(t)+\beta_{21}\Lambda_{2}(t),
\end{equation*}
\begin{equation*}
\Lambda_{1}^{2}(t)\Lambda_{2}(t) = \beta_{12}\Lambda_{1}^{2}(t) + \beta_{12}\beta_{21}\Lambda_{1}(t) + \beta_{21}^{2}\Lambda_{2}(t)
\end{equation*}
and 
\begin{equation*}
\Lambda_{1}(t)\Lambda_{2}^{2}(t) = \beta_{12}^{2}\Lambda_{1}(t) + \beta_{21}\Lambda_{2}^{2}(t) + \beta_{12}\beta_{21}\Lambda_{2}(t).
\end{equation*}
Substituting these identities, we can write $f(\mathbf{x}; t)$ in terms of \eqref{eq:genCanPol} as
\begin{eqnarray*}
f(\mathbf{x};t) & = & h_{11}(\mathbf{x})\Lambda_{1}(t)\lambda_{3}^{2}+h_{12}(\mathbf{x})\Lambda_{1}^{2}(t)\lambda_{3}+h_{13}(\mathbf{x})\Lambda_{1}^{3}(t)+\nonumber \\
 &  & h_{21}(\mathbf{x})\Lambda_{2}(t)\lambda_{3}^{2}+h_{22}(\mathbf{x})\Lambda_{2}^{2}(t)\lambda_{3}+h_{23}(\mathbf{x})\Lambda_{2}^{3}(t)+\lambda_{3}^{3}-\mu_{i}(t),
\end{eqnarray*}
where
\begin{eqnarray*}
h_{11}(\mathbf{x}) & = & x_{11}+x_{21}+x_{31}+\frac{\beta_{12}\beta_{21}}{\lambda_{3}^{2}}\left(x_{11}x_{21}x_{32}+x_{11}x_{22}x_{31}+x_{12}x_{21}x_{31}\right)+\\
 &  & \frac{\beta_{12}^{2}}{\lambda_{3}^{2}}\left(x_{11}x_{22}x_{32}+x_{12}x_{21}x_{32}+x_{12}x_{22}x_{31}\right)+\\
 &  & \frac{\beta_{12}}{\lambda_{3}}\left(x_{11}x_{22}+x_{11}x_{32}+x_{12}x_{21}+x_{21}x_{32}+x_{12}x_{31}+x_{22}x_{31}\right),
\end{eqnarray*}
\begin{equation*}
h_{12}(\mathbf{x)}=x_{11}x_{21}+x_{11}x_{31}+x_{21}x_{31}+\frac{\beta_{12}}{\lambda_{3}}\left(x_{11}x_{21}x_{32}+x_{11}x_{22}x_{31}+x_{12}x_{21}x_{31}\right),
\end{equation*}
\begin{equation*}
h_{13}(\mathbf{x})=x_{11}x_{21}x_{31},
\end{equation*}
and so on.
\end{example} \vspace{0.1in}

\noindent \emph{Step3-Eliminate dependence on $\tau$}: The advantage of \eqref{eq:genCanPol} is that, apart from $c(t),$ the number of t-dependent coefficients equals $d\cdot N_{i}$ which is exactly the number of unknowns in the polynomial $f$. Further, as shown below, they are linearly independent.

\begin{lemma} \label{lem:genTtauMatrix}
For $k \in [d\cdot N_i],$ let $b_k := \min \{ j \in [d]:j\cdot N_i \geq k \}.$ Then the matrix $T_{\tau}$, where, for $j,k \in [d \cdot N_i]$
\begin{equation}
(T_{\tau})_{jk} = \Lambda_{b_k}^{k - (b_k - 1)\cdot N_i}(t_j) \lambda_{d + 1}^{b_k \cdot N_i - k}, \label{eq:coeffMat}
\end{equation}
is non-singular.
\end{lemma}
\begin{proof}
See Appendix~\ref{sec: Apdx NonSingOfCMat}. \qquad
\end{proof}

Observe that if we let $\mathbf{c_{\tau}} \equiv (c(t_{1}),\ldots,c(t_{d\cdot N_{i}})),$ $\mathcal{E}_k(\mathbf{x}) \equiv (h_{k1}(\mathbf{x}), \ldots, h_{kN_{i}}(\mathbf{x}))$ and $\mathcal{E}(\mathbf{x})$ $\equiv (\mathcal{E}_1(\mathbf{x}),\ldots,$ $\mathcal{E}_d(\mathbf{x})),$ then \eqref{eq:tempPolSys} can be equivalently expressed as
\begin{equation}
\label{eq:genSymPolSys}
T_{\tau}\mathcal{E}(\mathbf{x})-\mathbf{c}_{\tau}=0.
\end{equation}
Premultiplying \eqref{eq:genSymPolSys} by $\left(T_{\tau}\right)^{-1},$ which now exists by Lemma~\ref{lem:genTtauMatrix}, we have
\begin{equation} \label{eq:genSecSymPolSys}
\mathcal{E}(\mathbf{x})-\left(T_{\tau}\right)^{-1}\mathbf{c}_{\tau}=0.
\end{equation}
Clearly, $\mathbf{w}\equiv(\mathbf{w}_{1},\ldots,\mathbf{w}_{N_{i}})$ is a root of \eqref{eq:tempGenPoly} and hence of \eqref{eq:genSecSymPolSys}. This immediately implies that $\left(T_{\tau}\right)^{-1}\mathbf{c}_{\tau} = \mathcal{E}(\mathbf{w})$ and consequently \eqref{eq:genSecSymPolSys} can rewritten as
\begin{equation} \label{eq:genElePolSys}
H_{i}(\mathbf{x})\equiv\mathcal{E}(\mathbf{x})-\mathcal{E}(\mathbf{w})=0.
\end{equation}
Note that \eqref{eq:genElePolSys} is devoid of any reference to the set $\tau$ and can be arrived at using any valid $\tau.$ For this reason, we will henceforth refer to \eqref{eq:genElePolSys} as the EPS.

\begin{example}
Let $N_i = d = 2.$ Also, let $\lambda_1 = 5, \lambda_2 = 3$ and $\lambda_3 = 1.$ Then the map $\mathcal{E}$ described above is given by
\begin{equation}
\label{eq:egEPS}
\mathcal{E}(\mathbf{x})=\left(
\begin{array}{c}
x_{11} + x_{21} + 5(x_{11}x_{22} + x_{12}x_{21})\\

x_{11}x_{21}\\

x_{12} + x_{22} - 6(x_{11}x_{22} + x_{12}x_{21})\\

x_{12}x_{22}\\
\end{array}
\right).
\vspace{0.2cm}
\end{equation}

\end{example}

Two useful properties of the EPS are stated next. For $\sigma \in S_{N_i},$ let $\mathbf{x}_\sigma := (\mathbf{x}_{\sigma(1)}, \ldots, \mathbf{x}_{\sigma(N_i)})$ denote a permutation of the vectors $\mathbf{x}_1, \ldots, \mathbf{x}_{N_i}.$ Further, for any $\mathbf{x} \in \mathbb{R}^{d \cdot N_i},$ let $\pi_{\mathbf{x}} := \{\mathbf{x}_{\sigma}: \sigma \in S_{N_i}\}.$ 

\begin{lemma} \label{lem:EPSSymm}
$H_i(\mathbf{x}) = H_i(\mathbf{x}_\sigma),$ $\forall \sigma \in S_{N_i}.$ That is, the map $H_i$ is symmetric.
\end{lemma}
\begin{proof}
Observe that $H_i(\mathbf{x}) = (T_\tau)^{-1} F_{\tau}(\mathbf{x})$ and $F_\tau,$ as defined in \eqref{eq:tempPolSys}, is symmetric. The result thus follows. \qquad \end{proof}

\begin{lemma} \label{lem:EPSWellBeh}
There exists an open dense set $\mathcal{R}_i$ of $\mathbb{R}^{d \cdot N_i}$ such that if $\mathbf{w} \in \mathcal{R}_i,$ then $|V(H_i)| = k \times N_i!,$ where $k \in \mathbb{Z}_{++}$ is independent of $\mathbf{w} \in \mathcal{R}_i.$ Further, each solution is non-singular.
\end{lemma}
\begin{proof}
See Appendix~\ref{sec: Apdx RPEPS}. \qquad \end{proof}

We henceforth assume that $\mathbf{w} \in \mathcal{R}_i.$ From the definition of the EPS in \eqref{eq:genElePolSys}, it is obvious that $\mathbf{w} \in V(H_i).$ By Lemma~\ref{lem:EPSSymm}, it also follows that if $\mathbf{x}^{*} \in V(H_i),$ then $\pi_{\mathbf{x}^{*}} \subset V(H_i).$ Hence, it suffices to  work with
\begin{equation} \label{eq:redSolSet}
\mathcal{M}_i = \{\alpha \in \mathbb{C}^d : \exists \mathbf{x}^* \in V(H_i) \mbox{ with } \mathbf{x}_1^* = \alpha \}.
\end{equation}
Clearly, $\mathcal{W}_i := \{\mathbf{w}_1, \ldots, \mathbf{w}_{N_i}\} \subset \mathcal{M}_i.$ A point to note here is that $\mathcal{I}_i := \mathcal{M}_i \backslash \mathcal{W}_i$ is not empty in general.

Our next objective is to develop the above theory for the case where for each $i \in [m],$ instead of the exact value of $M_{Y_i}(t),$ we have access only to the IID realizations $\{Y_{il}\}_{l \geq 1}$ of the random variable $Y_i.$ That is, for each $k \in [N_i],$ we have to use the sample average $\hat{M}_{Y_i}(t_k; L) = \left(\sum_{l = 1}^{L} \exp(-t_k Y_{il})\right)/L $ for an appropriately chosen large $L,$ $\hat{c}(t_k; L) = \hat{M}_{Y_i}(t_k)(\lambda_{d+1} + t_k)^{N_i} - \lambda_{d+1}^{N_i}$ and $\hat{\mathbf{c}}_{\tau, L} \equiv (\hat{c}(t_1; L), \ldots, \hat{c}(t_{d \cdot N_i}; L))$ as substitutes for each $M_{Y_i}(t_k),$ each $c(t_k)$ and $\mathbf{c}_\tau$ respectively. But even then note that the noisy or the perturbed version of the EPS
\begin{equation}
\label{eq:NoiGenElePolSys}
\hat{H}_i(\mathbf{x}) \equiv \mathcal{E}(\mathbf{x}) - (T_{\tau})^{-1}(\hat{\mathbf{c}}_{\tau, L}) = 0.
\end{equation}
is always well defined. More importantly, the perturbation is only in its constant term. As in Lemma~\ref{lem:EPSSymm}, it then follows that the map $\hat{H}_i$ is symmetric.

Next observe that since $\mathcal{R}_i$ is open (see Lemma~\ref{lem:EPSWellBeh}), there exists a small enough $\bar{\delta}_i > 0$ such that $B(\mathbf{w}; \bar{\delta}_i) \subset \mathcal{R}_i.$ Using the regularity of solutions of the EPS (see Lemma~\ref{lem:EPSWellBeh}), the inverse function theorem then gives the following result.

\begin{lemma}
\label{lem:NoiEPSWellBeh}
Let $\delta \in (0, \bar{\delta}_i)$ be such that for any two distinct solutions in $V(H_i),$ say $\mathbf{x}^{*}$ and $\mathbf{y}^{*},$ $B(\mathbf{x}^{*}; \delta) \cap B(\mathbf{y}^{*}; \delta) = \emptyset.$ Then there exists an $\epsilon(\delta) > 0$ such that if $\mathbf{u} \in \mathbb{R}^{d \cdot N_i}$ and $||\mathbf{u} - \mathcal{E}(\mathbf{w})|| < \epsilon(\delta),$ then the solution set $V(\hat{H}_i)$ of the perturbed EPS $\mathcal{E}(\mathbf{x}) - \mathbf{u} = 0$ satisfies the following:
\begin{enumerate}
\item All roots in $V(\hat{H}_i)$ are regular points of the map $\mathcal{E}.$

\item For each $\mathbf{x}^{*} \in V(H_i),$ there is one and only one $\mathbf{z}^{*} \in V(\hat{H}_i)$ such that $||\mathbf{x}^{*} - \mathbf{z}^{*}|| < \delta.$
\end{enumerate}
\end{lemma}

As a consequence, we have the following.
\begin{lemma}
\label{lem:ExistL}
Let $\delta$ and $\epsilon(\delta)$ be as described in Lemma~\ref{lem:NoiEPSWellBeh}. Then for tolerable failure rate $\kappa > 0$ and the chosen set $\tau,$ $\exists L_{\tau, \delta, \kappa} \in \mathbb{Z}_{++}$ such that if $L \geq L_{\tau, \delta, \kappa},$ then with probability greater than $1 - \kappa,$ we have $|| (T_\tau)^{-1} \hat{\mathbf{c}}_{\tau, L} - \mathcal{E}(\mathbf{w}) || < \epsilon(\delta).$
\end{lemma}
\begin{proof}
Note that $\exp(-t_k Y_{il}) \in [0,1]$ $\forall i, l$ and $k.$ The Hoeffding inequality (see \cite{Hoeffding63}) then shows that for any $\epsilon > 0$, $\Pr\{|\hat{M}_{Y_i}(t_k; L) - M_{Y_i}(t_k)| > \epsilon\} \leq \exp(-2 \epsilon^2 L).$ Since $\mathcal{E}(\mathbf{w}) = (T_\tau)^{-1}\mathbf{c}_\tau,$ the result is now immediate. \qquad \end{proof}

The above two results, in simple words, state that solving \eqref{eq:NoiGenElePolSys} for a large enough $L,$ with high probability, is almost as good as solving the EPS of \eqref{eq:genElePolSys}. For $L \geq L_{\tau, \delta, \kappa},$ let $\mathcal{A}_i(\kappa)$ denote the event $|| (T_\tau)^{-1} \hat{\mathbf{c}}_{\tau, L} - \mathcal{E}(\mathbf{w}) || < \epsilon(\delta).$ Clearly, $\Pr\{\mathcal{A}^{c}_i(\kappa)\} \leq \kappa.$ As in \eqref{eq:redSolSet}, let
\begin{equation}
\hat{\mathcal{M}}_i = \{\hat{\alpha} \in \mathbb{C}^d : \exists \mathbf{z}^{*} \in V(\hat{H}_i) \mbox{ with } \mathbf{z}_1^{*} = \alpha\}.
\end{equation}

We are now done discussing the EPS for an arbitrary $i \in [m].$ In summary, we have managed to obtain a set $\hat{\mathcal{M}}_i$ in which a close approximation of the weight vectors of random variables ${X_j}$ that add up to give $Y_i$ are present with high probability. The next subsection takes a unified view of the solution sets $\{\hat{\mathcal{M}}_i: i\in [m]\}$ to match the weight vectors to the corresponding random variables. But before that, we redefine $\mathcal{W}_i$ as $\{\mathbf{w}_j: j \in p_i\}.$ Accordingly, $\mathcal{M}_i, \hat{\mathcal{M}}_i, V(H_i)$ and $V(\hat{H}_i)$ are also redefined using notations of Section~\ref{sec:model}.

\subsection{Parameter matching using 1-identifiability}
\label{subsec:parMatch}

We begin by giving a physical interpretation for the $1-$identifiability condition of the matrix $A.$ For this, let $\mathcal{G}_{j} := \left\{ i\in[m]:j\in p_{i}\right\} $ and $\mathcal{B}_{j} := [m] \backslash \mathcal{G}_{j}.$
\begin{lemma}
  \label{lem:linkAsPathIntersection}
  For a $1-$identifiable matrix $A$, each index $j\in[N]$ satisfies
  \[
  \{j\} = \bigcap_{g\in\mathcal{G}_{j}} p_{g} \cap \bigcap_{b\in\mathcal{B}_{j}} p_{b}^{c} =: \mathcal{D}_{j}.
  \]
\end{lemma}
\begin{proof}
By definition, $j \in \mathcal{D}_{j}.$ For converse, if $k\in\mathcal{D}_{j},$ $k \neq j,$ then columns $j$ and $k$ of $A$ are identical; contradicting its $1-$identifiability condition. Thus $\{j\}=\mathcal{D}_{j}.$ \qquad \end{proof}

An immediate result is the following.

\begin{corollary}
  \label{cor:uniqueValuePathValueIntersection}
  Suppose $A$ is a $1-$identifiable matrix. If the map $u: [N] \rightarrow X,$ where $X$ is an arbitrary set, is bijective and $\forall i \in [m],$ $v_{i} := \{u(j):j\in p_{i}\},$ then for each $j \in [N]$
  \[
  \{u(j)\} =  \bigcap_{g\in\mathcal{G}_{j}} v_{g} \cap \bigcap_{b\in\mathcal{B}_{j}} v_{b}^{c}
  \]
\end{corollary}

By reframing this, we get the following result.

\begin{theorem}
  \label{thm:idlCasParMch}
Suppose $A$ is a $1-$identifiable matrix. If the weight vectors $\mathbf{w}_1, \ldots, \mathbf{w}_N$ are pairwise distinct, then the rule
  \begin{equation}
    \psi:j \rightarrow \bigcap_{g\in\mathcal{G}_{j}} \mathcal{W}_{g} \cap \bigcap_{b\in\mathcal{B}_{j}} \mathcal{W}_{b}^{c},
    \label{eq:idlCasAssRule}
  \end{equation}
satisfies $\psi(j) = \mathbf{w}_{j}$.
\end{theorem}

This result is where the complete potential of the $1-$ identifiability condition of $A$ is being truly taken advantage of. What this states is that if we had access to the collection of sets $\{\mathcal{W}_i: i \in [m]\},$ then using $\psi$ we would have been able to uniquely match the weight vectors to the random variables. But note that, at present, we have access only to the collection  $\{\mathcal{M}_i: i \in [m]\}$ in the ideal case and $\{\hat{\mathcal{M}}_i: i \in [m]\}$ in the perturbed case. In spite of this, we now show that if $\forall i_1, i_2 \in [m],$ $i_1 \neq i_2,$
\begin{equation} \label{eq:inValSolNotRel}
\mathcal{I}_{i_1} \cap \mathcal{M}_{i_2} = \emptyset, 
\end{equation}
a condition that always held in simulation experiments, then the rules:
\begin{equation}
  \label{eq:idlCasGenAssRule}
	\psi:j \rightarrow \bigcap_{g\in\mathcal{G}_{j}} \mathcal{M}_{g} \cap \bigcap_{b\in\mathcal{B}_{j}} \mathcal{M}_{b}^{c}
\end{equation}
for the ideal case, and
\begin{equation}
  \label{eq:PerCasGenAssRule}
	\hat{\psi}:j \rightarrow \bigcap_{g\in\mathcal{G}_{j}} \hat{\mathcal{M}}_{g} \cap \bigcap_{b\in\mathcal{B}_{j}} \hat{\mathcal{M}}_{b}^{c}
\end{equation}
in the perturbed case, with minor modifications recover the correct weight vector associated to each random variable $X_j.$

We first discuss the ideal case. Let $\mathcal{S}:=\{j \in [N]: |\mathcal{G}_j| \geq 2\}.$ Because of \eqref{eq:inValSolNotRel} and Theorem~\ref{thm:idlCasParMch}, note that
\begin{enumerate}
\item If $j \in \mathcal{S},$  then $\psi(j) = \{\mathbf{w}_j\}.$

\item If $j \in \mathcal{S}^{c}$ and $j \in p_{i^{*}},$ then $\psi(j) = \{\mathbf{w}_j\} \cup \mathcal{I}_{i^*}.$
\end{enumerate}
That is, \eqref{eq:idlCasGenAssRule} works perfectly fine when $j \in \mathcal{S}.$ The problem arises only when $j \in \mathcal{S}^{c}$ as $\psi(j)$ does not give as output a unique vector. To correct this, fix $j \in \mathcal{S}^{c}.$ If $j \in p_{i^*},$ then let $\mathbf{v}^{sub} \equiv (\mathbf{w}_k : k \in p_{i^*} \backslash \{j\}).$ Because of $1-$identifiability, note that if $k \in p_{i^*}\backslash \{j\},$ then $k \in \mathcal{S}.$ From \eqref{eq:tempGenPoly} and \eqref{eq:tempPolSys}, it is also clear that $(\mathbf{v}^{sub}, \alpha) \in V(H_i)$ if and only if $\alpha = \mathbf{w}_j.$ This suggests that we need to match parameters in two stages. In stage 1, we use \eqref{eq:idlCasGenAssRule} to assign weight vectors to all those random variables $X_j$ such that $j \in \mathcal{S}.$ In stage 2, for each $j \in \mathcal{S}^{c},$ we identify $i^* \in [m]$ such that $j \in p_{i^*}.$  We then construct $\mathbf{v}^{sub}.$ We then assign to $j$ that unique $\alpha$ for which $(\mathbf{v}^{sub}, \alpha) \in V(H_{i^*}).$ Note that we are ignoring the trivial case where $|p_{i^*}| = 1.$ It is now clear that by using \eqref{eq:idlCasGenAssRule} with modifications as described above, at least for the ideal case, we can uniquely recover back for each random variable $X_j$ its corresponding weight vector $\mathbf{w}_j.$

We next handle the case of noisy measurements. Let $\mathcal{U}:= \cup_{i \in [m]} \mathcal{M}_i$ and $\hat{\mathcal{U}} := \cup_{i \in [m]} \hat{\mathcal{M}}_i.$ Observe that using \eqref{eq:PerCasGenAssRule} directly, with probability one, will satisfy $\hat{\psi}(j) = \emptyset$ for each $j \in [N].$ This happens because we are distinguishing across the solution sets the estimates obtained for a particular weight vector. Hence as a first step we need to define a relation $\sim$ on $\hat{\mathcal{U}}$ that associates these related elements. Recall from Lemmas~\ref{lem:NoiEPSWellBeh} and \ref{lem:ExistL} that the set $\hat{\mathcal{M}}_i$ can be constructed for any small enough choice of $\delta, \kappa > 0.$ With choice of $\delta$ that satisfies
\begin{equation}
\label{eq:delCond}
0 < 4\delta < \underset{\alpha, \beta \in \mathcal{U}}{\min} ||\alpha - \beta||,
\end{equation}
let us consider the event $\mathcal{A} : = \cap_{i \in [m]} \mathcal{A}_i(\kappa/m).$  Using a simple union bound, it follows that $\Pr\{\mathcal{A}^{c}\} \leq \kappa.$ Now suppose that the event $\mathcal{A}$ is a success. Then by \eqref{eq:delCond} and Lemma~\ref{lem:NoiEPSWellBeh}, the following observations follow trivially.

\begin{enumerate}
\item For each $i \in [m]$ and each $\alpha \in \mathcal{M}_i,$ there exists at least one $\hat{\alpha} \in \hat{\mathcal{M}}_i$ such that $||\hat{\alpha} - \alpha|| < \delta.$

\item For each $i \in [m]$ and each $\hat{\alpha} \in \hat{\mathcal{M}}_i,$ there exists precisely one $\alpha \in \mathcal{M}_i$ such that $||\hat{\alpha} - \alpha|| < \delta.$

\item Suppose for distinct elements $\alpha, \beta \in \mathcal{U},$ we have $\hat{\alpha}, \hat{\beta} \in \hat{\mathcal{U}}$ such that $||\hat{\alpha} - \alpha|| < \delta$ and $||\hat{\beta} - \beta|| < \delta.$ Then $||\hat{\alpha} - \hat{\beta}|| > 2\delta.$
\end{enumerate}

From these, it is clear that the relation $\sim$ on $\hat{\mathcal{U}}$ should be
\begin{equation}
\label{eq:Rel}
  \hat{\alpha} \sim \hat{\beta} \mbox{ iff } ||\hat{\alpha} - \hat{\beta}|| < 2\delta.
\end{equation}
It is also easy to see that, whenever the event $\mathcal{A}$ is a success, $\sim$ defines an equivalence relation on $\hat{\mathcal{U}}.$ For each $i \in [m],$ the obvious idea then is to replace each element of $\hat{\mathcal{M}}_i$ and its corresponding $d-$ dimensional component in $V(\hat{H}_i)$ with its equivalence class. It now follows that \eqref{eq:PerCasGenAssRule}, with modifications as was done for the ideal case, will satisfy
\begin{equation}
\hat{\psi}(j) = \{\hat{\alpha} \in \hat{\mathcal{U}}:||\hat{\alpha}- \mathbf{w}_j|| < \delta\}.
\end{equation}
This is obviously the best we could have done starting from the set $\{\hat{\mathcal{M}}_i: i \in [m]\}.$

We end this section by summarizing our complete method in an algorithmic fashion. 

\begin{algorithm} \textbf{\large Distribution tomography}
\label{alg:EstLinkDis}

\noindent \textit{\large Phase 1: Construct \& Solve the EPS.}

For each $i \in [m],$
\begin{enumerate}
\item Choose an arbitrary $\tau = \{t_1, \ldots, t_{d \cdot N_i}\}$ of distinct positive real numbers.
\item Pick a large enough $L \in \mathbb{Z}_{++}.$ Set $\hat{M}_{Y_{i}}(t_{j}) = \left(\sum_{l=1}^{L}\exp(-t_jY_{il})\right)/L,$ $\hat{\mu}_{i}(t_{j})=(\lambda_{d+1}+t_{j})^{N_{i}}\hat{M}_{Y_{i}}(t_{j})$ and $\hat{c}(t_{j})=\hat{\mu}_{i}(t_{j}) - \lambda_{d+1}^{N_{i}}$ for each $j \in [N_i].$  Using this, construct $\hat{c}_{\tau}\equiv(\hat{c}(t_{1}),\ldots,\hat{c}(t_{d\cdot N_{i}})).$
\item Solve $\mathcal{E}(\mathbf{x})-T_{\tau}^{-1}\hat{c}_{\tau}=0$ using any standard solver for polynomial systems.
\item Build $\hat{\mathcal{M}}_i = \{\alpha \in \mathbb{C}^d : \exists \mathbf{x}^* \in V(\hat{H}_i) \mbox{ with } \mathbf{x}_1^* = \alpha \}.$
\end{enumerate}

\noindent \textit{\large Phase 2: Parameter Matching}

\begin{enumerate}
\item Set $\hat{\mathcal{U}}:=\bigcup_{i\in[m]}\hat{\mathcal{M}}_{i}.$ Choose $\delta>0$ small enough and define the relation $\sim$ on $\mathcal{\hat{U}},$ where $\hat{\alpha} \sim \hat{\beta}$ if and only if $||\hat{\alpha} - \hat{\beta}||_{2} < 2\delta.$ If $\sim$ is not an equivalence relation, then choose a smaller $\delta$ and repeat.
\item Construct the quotient set $\hat{\mathcal{U}}\backslash\sim.$ Replace all elements of each $\hat{\mathcal{M}_{i}}$ and each $V(\hat{H}_{i})$ with their equivalence class.
\item For each $j\in\mathcal{S},$ set $\hat{\psi}(j)=\left(\bigcap_{g\in\mathcal{G}_{j}}\mathcal{\hat{M}}_{g}\right)\cap\left(\bigcap_{b\in\mathcal{B}_{j}}\hat{\mathcal{M}}_{b}^{c}\right).$
\item For each $j\in\mathcal{S}^{c},$

\begin{enumerate}
\item Set $i^{*} = i\in[m]$ such that $j\in p_{i^{*}}.$
\item Construct $\mathbf{v}^{sub} \equiv(\hat{\psi}(k):k\in p_{i^{*}}, k\ne j\}.$
\item Set $\hat{\psi}(j)=\hat{\alpha}$ such that $(\mathbf{v}^{sub},\hat{\alpha})\in V(\hat{H}_{i}).$

\end{enumerate}
\end{enumerate}
\end{algorithm}

\section{Universality}
\label{sec:universality}
The crucial step in the DT scheme described above was to come up with, for each $i \in [m],$ a well behaved polynomial system, i.e., one that satisfies the properties of Lemma~\ref{lem:EPSWellBeh}, based solely on the samples of $Y_i.$ Once that was done, the ability to match parameters to the component random variables was only a consequence of the $1-$identifiability condition of the matrix $A.$ This suggests that it may be possible to develop similar schemes even in settings different to the ones assumed in Section~\ref{sec:model}. In fact, functions other than the MGF could also serve as blueprints for constructing the polynomial system. We discuss in brief few of these ideas in this section. Note that we are making a preference for polynomial systems for the sole reason that there exist computationally efficient algorithms, see for example \cite{Sommese05, Li97, Morgan89, Verschelde99}, to determine all its roots.

Consider the case, where $\forall j \in [N],$ the distribution of $X_j$ is the finite mixture model 
\begin{equation} \label{eq:genMixDist}
F_j(u) = \sum_{k = 1}^{d_j + 1}w_{jk}\phi_{jk}(u),
\end{equation}
where $d_j \in \mathbb{Z}_{++},$ $w_{j1}, \ldots, w_{j(d_l + 1)}$ denote the mixing weights, i.e., $w_{jk} \geq 0$ and $\sum_{k = 1}^{d_j + 1}w_{jk} = 1,$ and $\{\phi_{jk}(u)\}$ are some basis functions, say Gaussian, uniform, etc. The MGF of each $X_j$ is clearly given by
\begin{equation}
M_{X_j}(t) = \sum_{k = 1}^{d_j + 1}w_{jk}\int_{u = 0}^{\infty}\exp(-ut)d\phi_{jk}(u).
\end{equation}
Now note that if the basis functions $\{\phi_{jk}\}$ are completely known, then the MGF of each $Y_i$ will again be a polynomial in the mixing weights, $\{w_{jk}\},$ similar in spirit to the relation of \eqref{eq:MGFRel}. As a result, the complete recipe of Section~\ref{sec:DistTomoSc} can again be attempted to estimate the weight vectors of the random variables $X_j$ using only the IID samples of each $Y_i.$

In relation to \eqref{eq:Dist} or \eqref{eq:genMixDist}, observe next that $\forall n \in \mathbb{Z}_{++},$ the $n^{th}$ moment of each $X_j$ is given by
\begin{equation}
\mathbb{E}(X_j^n) = \sum_{k = 1}^{d_j + 1}w_{jk}\int_{u = 0}^{\infty}u^n d\phi_{jk}(u).
\end{equation}
Hence, the $n^{th}$ moment of $Y_i$ is again a polynomial in the unknown weights. This suggests that, instead of the MGF, one could use the estimates of the moments of $Y_i$ to come up with an alternative polynomial system and consequently solve for the distribution of each $X_j.$

Moving away from the models of \eqref{eq:Dist} and \eqref{eq:genMixDist}, suppose that for each $j \in [N],$ $X_j \sim \exp(m_j).$ Assume that each mean $m_j < \infty$ and that  $m_{j_{1}} \neq m_{j_{2}}$ when $j_{1}\neq j_{2}.$ We claim that the basic idea of our method can be used here to estimate $m_1, \ldots, m_N$ and hence the complete distribution of each $X_j$ using only the samples of $Y_i.$ As the steps are quite similar when either i) we know $M_{Y_i}(t)$ for each $i \in [m]$ and every valid $t$ and ii) we have access only to the IID samples $\{Y_{il}\}_{l \geq 1}$ for each $i \in [m]$, we take up only the first case.

Fix $i \in [m]$ and let $p_i := \{j \in [N]: a_{ij} = 1\}.$ To simplify notations, let us relabel the random variables $\{X_j: j \in p_i\}$ that add up to give $Y_i$ as $X_1, \ldots, X_{N_i},$ where $N_i = |p_i|.$ Observe that the MGF of $Y_i,$ after inversion, satisfies
\begin{equation}
  \prod_{j=1}^{N_{i}} (1 + t m_{j}) = 1/M_{Y_{i}}(t).
  \label{eq:momGenRel}
\end{equation}
Using \eqref{eq:momGenRel}, we can then define the canonical polynomial
\begin{equation}
  f(\mathbf{x};t):=\prod_{j=1}^{N_{i}}(1+tx_{j})-c(t), \label{eq:canPol}
\end{equation}
where $\mathbf{x} \equiv (x_{1},\ldots,x_{N_{i}})$ and $c(t) = 1/M_{Y_{i}}(t).$ Now choose an arbitrary set $\tau = \{t_{1},\ldots,t_{N_{i}}\} \subset \mathbb{R}_{++}$ consisting of distinct numbers and define
\begin{equation}
  F_{\tau}( \mathbf{x}) \equiv (f_{1}(\mathbf{x}), \ldots,  f_{N_{i}}(\mathbf{x})) = 0, \label{eq:IntmdPolySys}
\end{equation}
where $f_{k}(\mathbf{x})=f(\mathbf{x};t_{k}).$ We emphasize that this system is square of size $N_{i},$ depends on the choice of subset $\tau$ and each polynomial $f_{k}$ is symmetric with respect to the variables $x_{1,}\ldots,x_{N_{i}}.$ In fact, if we let $\mathbf{c}_{\tau}\equiv(c(t_{1}),\ldots,c(t_{N_{i}}))$ and $\mathcal{E} (\mathbf{x}) \equiv (e_{1}(\mathbf{x}), \ldots, e_{N_{i}} (\mathbf{x})),$ where $e_{k}(\mathbf{x})=\sum_{1\leq j_{1}<j_{2}<\ldots<j_{k}\leq N_{i}}x_{j_{1}}\cdots x_{j_{k}}$ denotes the $k^{th}$ elementary symmetric polynomial in the $N_{i}$ variables $x_{1,}\ldots,x_{N_{i}},$ we can rewrite \eqref{eq:IntmdPolySys} as
\begin{equation}
  T_{\tau} \mathcal{E}(\mathbf{x}) - (\mathbf{c}_{\tau} - \mathbf{1}) = 0. \label{eq:firstSymPolSys}
\end{equation}
Here $T_{\tau}$ denotes a Vandermonde matrix of order $N_{i}$ in $t_{1},\ldots,t_{N_{i}}$ with $(T_{\tau})_{jk}=t_{j}^{k}.$ Its determinant, given by $\det(T_{\tau}) = \left(\prod_{j=1}^{n} t_{j} \right) \prod_{j>i}(t_{j}-t_{i}),$ is clearly non-zero. Premultiplying \eqref{eq:firstSymPolSys} by $T_{\tau}^{-1},$ we have
\begin{equation}
  \mathcal{E}(\mathbf{x}) - T_{\tau}^{-1}(\mathbf{c}_{\tau} -   \mathbf{1}) = 0. \label{eq:secSymPolSys}
\end{equation}
Observe now that the vector $\mathbf{m} \equiv (m_{1},\ldots,m_{N_{i}})$ is a natural root of \eqref{eq:IntmdPolySys} and hence of \eqref{eq:secSymPolSys}. Hence $T_{\tau}^{-1}(\mathbf{c}_{\tau} - \mathbf{1}) = \mathcal{E}(\mathbf{m}).$  The EPS for this case can thus be written as
\begin{equation}
  H_{i}(\mathbf{x}) \equiv \mathcal{E}(\mathbf{x}) - \mathcal{E}(\mathbf{m}) = 0. \label{eq:charPolSys}
\end{equation}

We next discuss the properties of this EPS, or more specifically, its solution set. For this, let $V(H_{i}):=\{\mathbf{x} \in \mathbb{C}^{N_i} : H_{i}(\mathbf{x}) = 0\}.$
\begin{lemma} \label{lem:EPSSolSet}
	$V(H_{i})=\pi_{\mathbf{m}}:= \{\sigma(\mathbf{m}): \sigma \in S_{N_{i}}\}.$
\end{lemma}
\begin{proof} This follows directly from \eqref{eq:charPolSys}. \qquad \end{proof}

\begin{lemma} \label{lem:EPSRegSol}
	For every $\mathbf{x}^{*}\in V(H_{i}),$ $\det(\dot{\mathcal{E}}(\mathbf{x^{*}})) \neq 0.$ \label{prop:regSol}
\end{lemma}
\begin{proof} This follows from the fact that $\det(\dot{\mathcal{E}}(\mathbf{x}))=\prod_{1\leq j<k\leq N_{i}}(x_{j}-x_{k}).$ \qquad \end{proof}

Because of Lemma~\ref{lem:EPSSolSet}, it suffices to work with only the first components of the roots. Hence we define
\begin{equation}
  \mathcal{M}_{i}:=\{\alpha^{*}\in\mathbb{C}: \exists \mathbf{x}^{*} \in V(H_i) \mbox{ with } \mathbf{x}^*_{1} = \alpha \},
\end{equation}
which in this case is equivalent to the set $\{m_1, \ldots, m_{N_i} \}.$ Reverting back to global notations, note that
\begin{equation}
  \mathcal{M}_{i}=\{m_j : j \in p_i\}. \label{eq:finPathSolSet}
\end{equation}

Since $i$ was arbitrary, we can repeat the above procedure to obtain the collection of solution sets $\{\mathcal{M}_i: i \in [m]\}.$ Arguing as in Theorem~\ref{thm:idlCasParMch}, it is now follows that if $A$ is $1-$identifiable, then the rule
\begin{equation}
\psi(j) = \bigcap_{g \in \mathcal{G}_j}\mathcal{M}_g \cap \bigcap_{b \in \mathcal{B}_j} \mathcal{M}_b^{c},
\end{equation}
where $\mathcal{G}_j = \{i \in [m]: j \in p_i\}$ and $\mathcal{B}_j = [m]\backslash\mathcal{G}_j,$ satisfies the relation $\psi(j) = m_j.$  That is, having obtained the sets $\{\mathcal{M}_i : i \in [m]\},$ one can use $\psi$ to match the parameters to the corresponding random variables.

This clearly demonstrates that even if a transformation of the MGF is a polynomial in the parameters to be estimated, our method may be applicable.

\section{Experimental Results} \label{sec:ExptResults}
We assess the performance of our DT scheme using matlab based simulation experiments. 

We consider the simplified network delay tomography setup wherein, given a sequence of end-to-end measurements of delay a probe packets experiences across a subset of paths in a network, we are required to estimate the delay distribution across each link. In particular, we suppose that the topology of the network is known a priori in the form of its path-link matrix, denoted $A \in \{0,1\}^{m \times N},$ and is unvarying during the measurement phase. The rows of $A$ correspond to the paths across which probe packets can be transmitted and its delay measured, while the columns correspond to individual links. Further, the element $a_{ij}$ is $1$ precisely when the $j^{th}$ link is present on the $i^{th}$ path. We let $X_j,$ a GH random variable, denote the probe packet delay across link $j$ and $Y_i$ the delay across path $i.$ We assume that the delay across different links are independent. If we let $X \equiv (X_1, \ldots, X_N)$ and $Y \equiv (Y_1, \ldots, Y_m),$ then observe that $Y = AX.$ Clearly, this setup now resembles the model of Section~\ref{sec:model}.

We simulate the networks given in Figure~\ref{fig:simNet1}. The specifics of each experiment and observations made are described next. Note that, unless specified otherwise, all values are rounded to 2 significant digits. 

\begin{figure}[t]
\begin{center}
\subfigure[Tree topology]
{
	\label{fig:tree}
  \includegraphics[width=0.35\columnwidth]{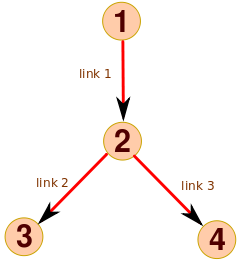}
}
\subfigure[General topology]
{
	\label{fig:general}
  \includegraphics[width=0.35\columnwidth]{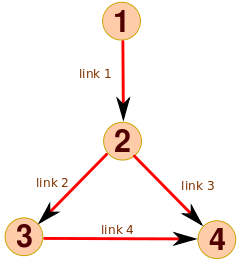}
}
\end{center} \label{fig:simNet1}
\caption{Networks for simulation experiments.}
\end{figure}

\subsection{Network with Tree Topology}
\label{subsec:treeTopology}
We work here with the network of Figure ~\ref{fig:tree}. Node $1$ is the source node, while nodes $3$ and $4$ act as sink. Path $p_1$ connects the nodes $1, 2$ and $3,$ while path $p_2$ connects the nodes $1, 2$ and $4.$ The path-link matrix is thus given by
\[
A = \left (
\begin{array}{ccc}
1 & 1 & 0\\

1 & 0 & 1\\
\end{array}
 \right).
\]

\begin{experiment} \label{exp:expt1}
We set the count of exponential stages in each link distribution to three. That is, we set $d = 2.$ We take the corresponding exponential stage parameters $\lambda_1, \lambda_2$ and $\lambda_3$ to be $5, 3$ and $1$ respectively. For each link, the weight associated with each exponential stage is set as given in the columns labeled $w_{j1}, \ldots, w_{j3}$ of Table~\ref{tab:expt1}.

\begin{table}[ht]
\caption{\label{tab:expt1}Actual and Estimated weights for Expt.~\ref{exp:expt1}}
\centering
\begin{tabular}{| c | c  c  c | c  c  c |}
\hline
Link & $w_{j1}$ & $w_{j2}$ & $w_{j3}$ & $\hat{w}_{j1}$ & $\hat{w}_{j2}$ & $\hat{w}_{j3}$\\
\hline 
1 & 0.17  & 0.80 & 0.03 &  0.15  & 0.82 & 0.02\\
 
2 & 0.13  & 0.47 & 0.40 &  0.15  & 0.46 & 0.39\\

3 & 0.80  & 0.15 & 0.05 &  0.79  & 0.15 & 0.06\\
\hline 
\end{tabular}
\vspace{0.1cm}
\end{table}

We first focus on path $p_1.$ Observe that its EPS is given by the map of \eqref{eq:egEPS}. We collect now a million samples of its end-to-end delay. Choosing an arbitrary set $\tau = \{1.9857, 2.3782, 0.3581, 8.8619\},$ we run the first phase of Algorithm~\ref{alg:EstLinkDis} to obtain $\hat{\mathcal{M}}_1.$ This set along with its ideal counterpart is given in Table~\ref{tab:expt11}.

\begin{table}[ht]
\caption{\label{tab:expt11} Solution set of the EPS for path $p_1$ in Expt.~\ref{exp:expt1}}
\centering
\begin{tabular}{| c | c  c |}
\hline
sol-ID & $\mathcal{M}_1$ & $\hat{\mathcal{M}}_1$\\
\hline 
1 & (0.1300, 0.4700)  & (0.1542, 0.4558)\\
 
2 & (0.1700, 0.8000)  & (0.1292, 0.8356)\\

3 & (3.8304, -2.8410)  & (3.8525, -2.8646)\\

4 & (0.1933, 0.7768)  & (0.2260, 0.7394)\\

5 & (0.1143, 0.4840)  & (0.0882, 0.5152)\\

6 & (0.0058, -0.1323)  & (0.0052, -0.1330)\\
\hline 
\end{tabular}
\vspace{0.1cm}
\end{table}

Similarly, by probing the path $p_2$ with another million samples and with $\tau = \{0.0842, 0.0870, 0.0305, 0.0344\},$ we determine $\hat{\mathcal{M}}_2.$ The sets $\mathcal{M}_2$ and $\hat{\mathcal{M}}_2$ are given in Table~\ref{tab:expt12}.

\begin{table}[ht]
\caption{\label{tab:expt12} Solution set of the EPS for path $p_2$ in Expt.~\ref{exp:expt1}}
\centering
\begin{tabular}{| c | c  c |}
\hline
sol-ID & $\mathcal{M}_2$ & $\hat{\mathcal{M}}_2$\\
\hline 
1 & (0.8000, 0.1500)  & (0.7933, 0.1459)\\
 
2 & (0.1660, 0.7775)  & (0.1720, 0.8095)\\

3 & (5.5623, -4.5638)  & (5.5573, -4.5584)\\

4 & (0.1700, 0.8000)  & (0.1645, 0.7669)\\

5 & (0.8191, 0.1543)  & (0.8296, 0.1540)\\

6 & (0.0245, -0.0263)  & (0.0246, -0.0259)\\
\hline 
\end{tabular}
\vspace{0.1cm}
\end{table}

To match the weight vectors to corresponding links, firstly observe that the minimum distance between $\hat{\mathcal{M}}_1$ and $\hat{\mathcal{M}}_2$ is 0.0502. Based on this, we choose $\delta = 0.03$ and run the second phase of Algorithm~\ref{alg:EstLinkDis}. The obtained results are given in the second half of Table~\ref{tab:expt1}. Note that the weights obtained for the first link are determined by taking a simple average of the solutions obtained from the two different paths. The norm of the error vector is $0.0443.$
\end{experiment}

\begin{experiment}
\label{exp:expt2}
Keeping other things unchanged as in the setup of experiment~\ref{exp:expt1}, we consider here four exponential stages in the distribution of each $X_j.$ The exponential stage parameters $\lambda_1, \lambda_2, \lambda_3$ and $\lambda_4$ equal $5, 4, 0.005$ and $1$ respectively. The corresponding weights are given in Table~\ref{tab:expt2}. But observe that the weights of the third stage is negligible for all three links. Because of this, we ignore its presence completely. That is, we consider $d = 2$ and then run Algorithm~\ref{alg:EstLinkDis}. The results obtained are given in the second half of Table~\ref{tab:expt2}. The norm of the error vector is $0.1843.$

This clearly demonstrates that the exponential stages, which are insignificant across all link distributions, can be ignored. 

\begin{table}[ht]
\caption{\label{tab:expt2}Actual and Estimated weights for Expt.~\ref{exp:expt2}}
\begin{center}
\begin{tabular}{|c|c c c c|c c c c|}
\hline 
Link & $w_{j1}$ & $w_{j2}$ & $w_{j3}$ & $w_{j4}$ & $\hat{w}_{j1}$ & $\hat{w}_{j2}$ & $\hat{w}_{j3}$ & $\hat{w}_{j4}$\tabularnewline
\hline 
1 & 0.71  & 0.20 & 0.0010 & 0.08 & 0.77  & 0.20 & 0 & 0.03\\

2 & 0.41  & 0.17 & 0.0015 & 0.41 & 0.38  & 0.18 & 0 & 0.44\\

3 & 0.15  & 0.80 & 0.0002 & 0.04 & 0.12  & 0.70 & 0 & 0.18\\
\hline 
\end{tabular}
\end{center}
\vspace{0.1cm}
\end{table}
\end{experiment}

\subsection{Network with General Topology}
\label{subsec:genTopology}

We deal here with the network of Figure~\ref{fig:general}. Nodes $1$ and $2$ act as source while nodes $3$ and $4$ act as sink. We consider here three paths. Path $p_1$ connects the nodes $1, 2$ and $3,$ path $p_2$ connects the nodes $1, 2$ and $4,$ while path $p_3$ connects the nodes $2, 3$ and $4.$ The path-link matrix is 
\[
A = \left(
\begin{array}{cccc}
1 & 1 & 0 & 0\\
1 & 0 & 1 & 0\\
0 & 1 & 0 & 1\\
\end{array}
\right).
\]
\begin{experiment}
\label{exp:expt3}
The values of $d, \lambda_1, \lambda_2, \lambda_3$ are set as in Experiment~\ref{exp:expt1}. By choosing again a million probe packets for each path, we run Algorithm~\ref{alg:EstLinkDis}. The actual and estimated weights are shown in Table~\ref{tab:expt3}.

\begin{table}[h!]
\caption{\label{tab:expt3}Actual and Estimated weights for Expt.~\ref{exp:expt3}}
\begin{center}
\begin{tabular}{|c|c c c| c c c|}
\hline 
Link & $w_{j1}$ & $w_{j2}$ & $w_{j3}$ & $\hat{w}_{j1}$ & $\hat{w}_{j2}$ & $\hat{w}_{j3}$\tabularnewline
\hline 
1 & 0.34  & 0.26 & 0.40 & 0.34  & 0.24 & 0.42\\

2 & 0.46  & 0.49 & 0.05 & 0.45  & 0.50 & 0.05\\

3 & 0.12  & 0.65 & 0.23 & 0.11  & 0.68 & 0.21\\

4 & 0.71  & 0.19 & 0.10 & 0.69  & 0.18 & 0.13\\
\hline 
\end{tabular}
\end{center}
\vspace{0.1cm}
\end{table}

The ease with which our algorithm can handle even networks that have non-tree topologies is clearly demonstrated in this experiment. 

\end{experiment}

\section{Discussion} \label{sec:Disc}
This paper took advantage of the properties of polynomial systems to develop a novel algorithm for the GNT problem. For an arbitrary matrix $A,$ which is $1-$identifiable, it demonstrated successfully how to accurately estimate the distribution of the random vector $X,$ with mutually independent components, using only IID samples of the components of the random vector $Y = AX.$ Translating to network terminology, this means that one can now address the tomography problem even for networks with arbitrary topologies using only pure unicast probe packet measurements. The fact that we need only the IID samples of the components of $Y$ shows that the processes to acquire these samples across different paths can be asynchronous. Another nice feature of this approach is that it can estimate the unknown link level performance parameters even when no prior information is available about the same.

\begin{appendix}
\section{Nonsingularity of coefficient matrix} \label{sec: Apdx NonSingOfCMat}
In this section, we prove three results which will together demonstrate the validity of Lemma~\ref{lem:genTtauMatrix}.

\begin{lemma} \label{lem:matWNonSingular}
Let $\omega_1, \omega_2 \in \mathbb{Z}_{++}.$ Further suppose that $\lambda_{1},\lambda_{2}$ and $t_{1},\ldots,t_{\omega_{1}+\omega_{2}}$ are strictly positive real numbers satisfying $\lambda_{1} \neq \lambda_{2}$ and $t_{k_{1}}\ne t_{k_{2}},$ if $k_{1}\ne k_{2}.$ Then the square matrix
\[
W=\left(\left.\begin{array}{ccc}
\frac{1}{(\lambda_{1}+t_{1})^{\omega_{1}}} & \ldots & \frac{1}{(\lambda_{1}+t_{1})}\\
\vdots & \ddots & \vdots\\
\frac{1}{(\lambda_{1}+t_{\omega_{1} + \omega_{2}})^{\omega_{1}}} & \ldots & \frac{1}{(\lambda_{1}+t_{\omega_{1} + \omega_{2}})}
\end{array}\right|\begin{array}{ccc}
\frac{1}{(\lambda_{2}+t_{1})^{\omega_{2}}} & \ldots & \frac{1}{(\lambda_{2}+t_{2})}\\
\vdots & \ddots & \vdots\\
\frac{1}{(\lambda_{2}+t_{\omega_{1} + \omega_{2}})^{\omega_{2}}} & \ldots & \frac{1}{(\lambda_{2}+t_{\omega_{1}+\omega_{2}})}
\end{array}\right)
\]
is non-singular. 
\end{lemma}
\begin{proof}
It suffices to show that 
\begin{equation} \label{eq:detW}
\det(W) = \left(\prod_{j=1}^{\omega_{1}+\omega_{2}} \prod_{k = 1}^{2}\frac{\left(\lambda_{2}-\lambda_{1}\right)^{\omega_{1}\omega_{2}}}{\left(\lambda_{k} + t_{j}\right)^{\omega_{k}}}\right) \left(\prod_{1\leq j_{1}<j_{2}\leq(\omega_{1}+\omega_{2})}\left(t_{j_{2}}-t_{j_{1}}\right)\right).
\end{equation}

As a first step, we perform on $W$ the row operations
\[
\mbox{row}_{j}\leftarrow\mbox{row}_{j}\times\left(\lambda_{1}+t_{j}\right)^{\omega_{1}}\left(\lambda_{2}+t_{j}\right)^{\omega_{2}}
\]
for each $j \in [\omega_{1} + \omega_{2}]$ to get the matrix $B.$ Note that for $j, k \in [\omega_1 + \omega_2],$ 
\begin{equation} \label{eq:BMatDefn}
\left(B\right)_{jk} = 
\begin{cases}
\left(\lambda_{1}+t_{j}\right)^{k-1}\left(\lambda_{2}+t_{j}\right)^{\omega_{2}}, & k \in [\omega_1]\\
\\
\left(\lambda_{1}+t_{j}\right)^{\omega_{1}}\left(\lambda_{2}+t_{j}\right)^{k - 1 - \omega_{1}}, & \mbox{otherwise}.
\end{cases}
\end{equation}
From the properties of determinants, it follows that 
\[
\det(W)=\left(\prod_{j = 1}^{\omega_{1} + \omega_{2}}\frac{1}{\left(\lambda_{1}+t_{j}\right)^{\omega_{1}}\left(\lambda_{2}+t_{j}\right)^{\omega_{2}}}\right)\det(B).
\]
To verify \eqref{eq:detW}, it only remains to show that 
\begin{equation} \label{eq:detB}
\det(B)=(\lambda_{2}-\lambda_{1})^{\omega_{1}\omega_{2}}\prod_{1\leq j_{1}<j_{2}\leq(\omega_{1} + \omega_{2})}\left(t_{j_{2}}-t_{j_{1}}\right).
\end{equation}
Towards this, our approach is to treat $\lambda_{2}$ as a variable and the other indeterminates, i.e., $\lambda_{1},$ $t_{1},\ldots,$ $t_{\omega_{1} + \omega_{2}}$ as constants and show that $g(\lambda_2)$ is a univariate polynomial of degree $\omega_1\omega_2$ with $\lambda_1$ as its sole root with multiplicity $\omega_1\omega_2.$

We now introduce some notations. Let $b_{k}(\lambda_{2})$ denote the $k^{th}$ column of $B$ and $b_{k}^{r_{k}}(\lambda_{2})$ its element-wise $r_{k}^{th}$ derivative with respect to $\lambda_2.$ We use $g(\lambda_{2})$ and $g^{n}(\lambda_{2})$ to refer respectively to $\det(B)$ and its $n^{th}$ derivative with respect to $\lambda_2.$ Lastly, for $\omega, n \in \mathbb{Z}_{++},$ we use $\Delta_{\omega, n}$ to represent the set of non-negative integer valued vector solutions of the equation $\alpha_1 + \cdots + \alpha_\omega = n.$

We now prove \eqref{eq:detB} by equivalently showing the following:
\begin{enumerate}[i.]
\item $\left.g^n\left(\lambda_2\right)\right|_{\lambda_2 = \lambda_1} = 0$ for each $0 \leq n < \omega_1\omega_2.$
\item $\left.g^{\omega_1\omega_2}\left(\lambda_2\right)\right|_{\lambda_2 = \lambda_1} = (\omega_1\omega_2)! \prod_{1 \leq j_{1} <j_{2}\leq(\omega_{1} + \omega_{2})}\left(t_{j_{2}}-t_{j_{1}}\right) \neq 0.$
\item For all $n > \omega_1\omega_2,$ $g^n(\lambda_2) \equiv 0.$
\end{enumerate}

Firstly note that, for any $n \in \mathbb{Z}_{++},$
\begin{equation} \label{eq:rthDerDetPolinL2}
g^{n}(\lambda_{2})=\sum_{\Delta_{\omega_{1} + \omega_{2},n}} \frac{n!}{r_1! \cdots r_{\omega_1 + \omega_2}!}\det\left(\left[\begin{array}{ccc}
b_{1}^{r_{1}}(\lambda_{2}) & \cdots & b_{\omega_{1} + \omega_{2}}^{r_{\omega_{1} + \omega_{2}}}(\lambda_{2})\end{array}\right]\right).
\end{equation}
Our strategy is to deal with 
$\det\left(\left[\begin{array}{ccc}b_{1}^{r_{1}}(\lambda_{2}) & \cdots & b_{\omega_{1} + \omega_{2}}^{r_{\omega_{1} + \omega_{2}}}(\lambda_{2})\end{array}\right]\right) = : \mu(\lambda_{2};\mathbf{r})$ 
for every possible pattern of the tuple $\mathbf{r}\equiv\left(r_{1},\ldots,r_{\omega_{1} + \omega_{2}}\right) \in \mathbb{Z}_{+}^{\omega_{1} + \omega_{2}}.$ Although the different possibilities are huge in number, the following observations, which follow directly from (\ref{eq:BMatDefn}), will reveal that in almost all combinations of $\mathbf{r}$ either a column becomes itself zero or is a scalar multiple of another column. In fact there is only one unique pattern where the determinant will have to be actually evaluated. 

\begin{enumerate}[1.]
\item The column $b_{\omega_{1}+1}(\lambda_{2})$ is a constant with respect to $\lambda_{2}.$ This implies that if $r_{\omega_{1}+1} \geq 1,$ then $b_{\omega_{1}+1}^{r_{\omega_{1}+1}}(\lambda_{2}) \equiv \mathbf{0}$ and consequently $\mu(\lambda_{2};\mathbf{r}) \equiv 0$ irrespective of what values $r_{k},$ $k \neq \omega_{1}+1$ take. Hence we need to focus on only those tuples $\mathbf{r}$ where $r_{\omega_{1}+1}=0.$ 

\item Suppose that $2 \leq k \leq \omega_{2}$ and $r_{\omega_{1}+1},$ $r_{\omega_{1} + 2},\ldots,$ $r_{\omega_{1} + k - 1}=0.$ Then 
\[
b_{\omega_{1}+k}^{r_{\omega_{1}+k}}(\lambda_{2})\equiv
\begin{cases}
\mathbf{0}, & r_{\omega_{1}+k} \geq k \vspace{0.05in}\\

\left(\prod\limits_{j=1}^{r_{\omega_{1}+k}}\left(k-j\right)\right)b_{\omega_{1}+k-r_{\omega_{1}+k}}(\lambda_{2}), &1\leq r_{\omega_{1} + k} \leq k-1.
\end{cases}
\]
In other words, if $r_{\omega_{1}+1},$ $r_{\omega_{1}+2},$ $\ldots,$ $r_{\omega_{1}+k-1} = 0$ and $r_{\omega_{1}+k} > 0,$ then $\mu(\lambda_{2};\mathbf{r}) \equiv 0.$ An  inductive argument on $k$ then shows that we need to deal with only those tuples where $r_{\omega_{1} + 1} = \cdots = r_{\omega_{1} + \omega_{2}} = 0.$ For the next 2 observations, we assume that this condition explicitly holds for the tuple under consideration. 

\item For column $\omega_{1}$ note that if $r_{\omega_{1}} \geq \omega_{2} + 1,$ then $b_{\omega_{1}}^{r_{\omega_{1}}}(\lambda_{2}) \equiv \mathbf{0}.$ But when $r_{\omega_{1}} = \omega_{2},$ then 
$b_{\omega_{1}}^{r_{\omega_{1}}}(\lambda_{2}) \equiv \omega_{2}!\left((\lambda_{1}+t_{1})^{\omega_{1}-1}, \ldots, \left(\lambda_{1} + t_{\omega_{1} + \omega_{2}}\right)^{\omega_{1}-1}\right).$
On the other hand, when $\lambda_{2} = \lambda_{1},$
\[
\left.b_{\omega_{1}}^{r_{\omega_{1}}}(\lambda_{2})\right|_{\lambda_{2}=\lambda_{1}} = 
\begin{cases}
\left.b_{\omega_{1} + \omega_{2}}(\lambda_{2})\right|_{\lambda_{2}=\lambda_{1}}, &r_{\omega_{1}}=0 \vspace{0.1in}\\

\left.\left(\prod\limits_{j=0}^{r_{\omega_{1}}-1}\left(\omega_{2}-j\right)\right)b_{\omega_{1} + \omega_{2}-r_{\omega_{1}}}(\lambda_{2})\right|_{\lambda_{2}=\lambda_{1}}, & 1\leq r_{\omega_{1}}\leq \omega_{2}-1.
\end{cases}
\]
This implies that $\mu(\lambda_{2};\mathbf{r}) \equiv 0$ if $r_{\omega_{1}} \geq \omega_{2}+1,$ while $\left.\mu(\lambda_{2};\mathbf{r})\right|_{\lambda_{2}=\lambda_{1}} = 0$  if $0 \leq r_{\omega_{1}} \leq \omega_{2} - 1.$ For our purposes, it thus remains to investigate only those tuples where $r_{\omega_{1}} = \omega_{2}.$ 

\item The following holds whenever $\omega_1 > 1.$ Suppose that there exists $k \in [\omega_{1} - 1]$ such that $r_{\omega_{1}} = r_{\omega_{1}-1} = \cdots = r_{\omega_{1} - k + 1} = \omega_{2}.$ It is then immediate that $b_{\omega_{1} - k}^{r_{\omega_{1} - k}}(\lambda_{2})\equiv \omega_{2}!\left((\lambda_{1}+t_{1})^{\omega_{1} - 1 -k}, \ldots, \left(\lambda_{1} + t_{\omega_{1} + \omega_{2}}\right)^{\omega_{1} - 1 -k}\right)$ if $r_{\omega_1 - k} = \omega_2,$ while $b_{\omega_{1}-k}^{r_{\omega_{1}-k}}(\lambda_{2}) \equiv 0$ if $r_{\omega_{1}-k} \geq \omega_{2} + 1.$ For the case when $r_{\omega_1 - k} \in \{0\} \cup [\omega_2 - 1],$ consider the following subcases.

\begin{enumerate}[(a)]
\item $\omega_1 \leq \omega_2:$ Here note that 
\[
\left.b_{\omega_{1}-k}^{r_{\omega_{1}-k}}(\lambda_{2})\right|_{\lambda_{2}=\lambda_{1}} = 
\begin{cases}
\left.b_{\omega_{1} + \omega_{2} - k}(\lambda_{2})\right|_{\lambda_{2}=\lambda_{1}}, & r_{\omega_{1}-k}=0 \vspace{0.15in}\\

\begin{array}{c}
\left(\prod\limits_{j = 0}^{r_{(\omega_{1} - k)} - 1}\left(\omega_{2} - j\right)\right) \times \\  \left.b_{\omega_{1} + \omega_{2} - k - r_{\omega_{1} - k}}(\lambda_{2})\right|_{\lambda_{2} = \lambda_{1}}
\end{array}, & 1 \leq r_{\omega_{1} - k} < \omega_{2} - k \vspace{0.15in}\\

\frac{\left.b_{\omega_{1} + \omega_{2} - k - r_{\omega_{1} - k}}^{\omega_{2}}(\lambda_{2})\right|_{\lambda_{2} = \lambda_{1}}}{\left(\omega_{2} - r_{\omega_{1} - k}\right)!}, & \omega_{2} - k\leq r_{\omega_{1} - k} < \omega_{2}.
\end{cases}
\]

\item $\omega_1 > \omega_2:$ If $k < \omega_2,$ then $b^{r_{\omega_1 - k}}_{\omega_1 - k}$ behaves precisely as in the subcase above. On the other hand, when $\omega_2 \leq k \leq \omega_1 - 1$ then
\[
\left.b_{\omega_{1}-k}^{r_{\omega_{1}-k}}(\lambda_{2})\right|_{\lambda_{2}=\lambda_{1}} = \frac{\left.b_{\omega_{1} + \omega_{2} - k - r_{\omega_{1} - k}}^{\omega_{2}}(\lambda_{2})\right|_{\lambda_{2} = \lambda_{1}}}{\left(\omega_{2} - r_{\omega_{1} - k}\right)!}, \hspace{0.12in} 0 \leq r_{\omega_{1} - k} < \omega_{2}. 
\]
\end{enumerate}

As a consequence of a simple induction on $k,$ it now follows that $\left.\mu(\lambda_{2};\mathbf{r})\right|_{\lambda_{2} = \lambda_{1}} = 0$ for all tuples $\mathbf{r},$ where $0 \leq r_{1} < \omega_{2}$ and for each $2 \leq i \leq \omega_{1},$ $r_{i} \in \{0\} \cup [\omega_2].$ In contrast, even if one amongst $r_{1},$ $\ldots,$
$r_{\omega_{1}}$ has value strictly bigger than $\omega_{2},$ then $\mu(\lambda_{2};\mathbf{r}) \equiv 0.$ 

\item The above observations essentially show that $\mathbf{r}^{*},$ where $r_{1}^{*} = \cdots = r_{\omega_{1}}^{*}= \omega_{2}$ and $r_{\omega_{1}+1}^{*} = \cdots = r_{\omega_{1}+\omega_{2}}^{*}=0,$ is the only tuple for which $\mu(\lambda_{2};\mathbf{r}^*)$ is required to be explicitly evaluated. Using properties of Vandermonde matrices, a simple calculation show that 
\[
\left.\mu(\lambda_{2};\mathbf{r}^{*})\right|_{\lambda_{2}=\lambda_{1}} = (\omega_{2}!)^{\omega_{1}}\prod_{1\leq k_{1}<k_{2}\leq \omega_{1}+\omega_{2}}\left(t_{k_{2}}-t_{k_{1}}\right) \neq 0.
\]
\end{enumerate}

\noindent Let $\mathcal{A}_{\leq,\mathbf{0}} = \left\{ \mathbf{r} \equiv (\mathbf{s},\mathbf{0}) \in \mathbb{Z}_{+}^{\omega_{1}} \times \mathbb{Z}_{+}^{\omega_{2}}:\,0\leq s_{1} < \omega_{2} \mbox{ and }0 \leq s_{k} \leq \omega_{2} \hspace{0.1in} \forall 2 \leq k \leq \omega_1 \right\},$ 
$\mathcal{A}_{\mathbf{r}^{*}}=\left\{ \mathbf{r}^{*}\right\}$ and
$\mathcal{A}_{>} = \mathbb{Z}_{+}^{\omega_{1}+\omega_{2}}\backslash\left\{ \mathcal{A}_{\mathbf{r}^{*}}\cup\mathcal{A}_{<,\mathbf{0}}\right\}$ be a partition of $\mathbb{Z}_{+}^{\omega_{1}+\omega_{2}}.$ Then a short summary of the above observations is that 
\[
\mu(\lambda_{2};\mathbf{r})\bigr|_{\lambda_{2} = \lambda_{1}} = 
\begin{cases}
0, & \mathbf{r} \in \mathcal{A}_{\leq, \mathbf{0}} \vspace{0.1in}\\

\left.\mu(\lambda_{2};\mathbf{r}^{*})\right|_{\lambda_{2}=\lambda_{1}}, & \mathbf{r}\in\mathcal{A}_{\mathbf{r}^{*}},
\end{cases}
\]
while $\mathbf{r}\in\mathcal{A}_{>}$ implies that $\mu(\lambda_{2};\mathbf{r}) \equiv 0.$ This last conclusion implies that (\ref{eq:rthDerDetPolinL2}) can be explicitly written as
\[
g^{n}(\lambda_{2})=\sum_{\mathbf{r}\in\Delta_{\omega_{1}+\omega_{2},n}\cap\{\mathcal{A}_{\mathbf{r}^{*}}\cup\mathcal{A}_{\leq,\mathbf{0}}\}}\frac{n!}{r_1! \cdots r_{\omega_1 + \omega_2}!} \mu(\lambda_{2};\mathbf{r}).
\]
Now note that $\forall \mathbf{r} \in\mathcal{A}_{\leq,\mathbf{0}},$ $\sum_{k=1}^{\omega_{1}+\omega_{2}}r_{k}<\omega_{1}\omega_{2},$ while $\sum_{k=1}^{\omega_{1}+\omega_{2}}r_{k}^{*}=\omega_{1}\omega_{2}.$ This shows that $\Delta_{\omega_{1} + \omega_{2},n} \cap \mathcal{A}_{\leq, \mathbf{0}} = \emptyset$ for all $n \geq \omega_{1}\omega_{2},$ while $\Delta_{\omega_{1} + \omega_{2},n} \cap \mathcal{A}_{\mathbf{r}^{*}} = \emptyset$ for all $n\neq \omega_{1}\omega_{2}.$ 

It is now trivial to see that $\left.g^{n}(\lambda_{2})\right|_{\lambda_{2}=\lambda_{1}}=0$ $\forall 0 \leq n < \omega_{1}\omega_{2},$ $\left.g^{\omega_{1}\omega_{2}}(\lambda_{2})\right|_{\lambda_{2} = \lambda_{1}} = (\omega_{1}\omega_{2})!\prod_{1\leq k_{1}<k_{2}\leq \omega_{1}+\omega_{2}}\left(t_{k_{2}}-t_{k_{1}}\right)$ and $g^{n}(\lambda_{2}) \equiv 0$ for all $n\geq \omega_{1}\omega_{2} + 1.$ This establishes the desired result. \qquad
\end{proof}

The general version of the above result is the following.
\begin{lemma} \label{lem:genMatWNonSingular}
For a fixed integer $d\geq2,$ let $\omega_{1},$ $\ldots,$ $\omega_{d}$ be arbitrary natural numbers. Let $S_{0}=0$ and for $k \in [d],$ $S_{k}=S_{k - 1} + \omega_{k}.$ Further suppose that $\lambda_{1},$ $\ldots,$ $\lambda_{d}$ and $t_{1},$ $\ldots,$ $t_{S_{d}}$ are strictly positive real numbers satisfying $\lambda_{i_{1}}\neq\lambda_{i_{2}},$ if $i_{1}\neq i_{2}$ and $t_{j_{1}}\ne t_{j_{2}},$ if $j_{1}\ne j_{2}.$
Then the square matrix
\begin{equation}
W = \left(\left.\begin{array}{ccc}
\frac{1}{(\lambda_{1}+t_{1})^{\omega_{1}}} & \ldots & \frac{1}{(\lambda_{1}+t_{1})}\\
\vdots & \ddots & \vdots\\
\frac{1}{(\lambda_{1}+t_{S_{d}})^{\omega_{1}}} & \ldots & \frac{1}{(\lambda_{1}+t_{\mathcal{S}_{d}})}
\end{array}\right|\cdots\left|\begin{array}{ccc}
\frac{1}{(\lambda_{d}+t_{1})^{\omega_{d}}} & \ldots & \frac{1}{(\lambda_{d}+t_{1})}\\
\vdots & \ddots & \vdots\\
\frac{1}{(\lambda_{d}+t_{S_{d}})^{\omega_{d}}} & \ldots & \frac{1}{(\lambda_{d}+t_{S_{d}})}
\end{array}\right.\right)\label{eq:genMatW}
\end{equation}
is non-singular. 
\end{lemma}
\begin{proof}
It suffices to show that 
\begin{equation} \label{eq:genDetA}
\det(W) = \left(\prod_{j=1}^{S_{d}}\prod_{i=1}^{d}\frac{\prod\limits_{(i_{1},i_{2})\in\mathcal{I}_d}(\lambda_{i_{2}}-\lambda_{i_{1}})^{\omega_{i_{1}}\omega_{i_{2}}}}{\left(\lambda_{i}+t_{j}\right)^{\omega_{i}}}\right)\prod_{(j_{1},j_{2})\in\mathcal{J}}(t_{j_{2}}-t_{j_{1}}),
\end{equation}
where $\mathcal{I} _d = \left\{ (i_{1},i_{2}):\, i_{1},i_{2}\in[d],\, i_{1}<i_{2}\right\} $ and $\mathcal{J}=\left\{ (j_{1},j_{2}):\, j_{1},j_{2}\in[S_{d}],\, j_{1}<j_{2}\right\} .$

As a first step, we do the row operations
\[
\mbox{row}_{j}\leftarrow\mbox{row}_{j}\times\prod_{i=1}^{d}\left(\lambda_{i}+t_{j}\right)^{\omega_{i}}
\]
for each $j \in [S_{d}]$ on $W$ to get the matrix $B.$ For $j, k \in [S_d],$ note that 
\begin{equation} \label{eq:BStruct}
(B)_{jk} = 
\begin{cases}
\left\{ \prod\limits_{i\in[d]\backslash\{1\}}\left(\lambda_{i}+t_{j}\right)^{\omega_{i}}\right\} \left(\lambda_{1}+t_{j}\right)^{k- S_0 - 1}, & S_{0} + 1 \leq k \leq S_{1}\\
\vdots\\
\left\{\prod\limits_{i \in [d] \backslash \{d\}} \left(\lambda_{i} + t_{j}\right)^{\omega_{i}} \right\} \left(\lambda_{d} + t_{j}\right)^{k - S_{(d - 1)} - 1}, & S_{(d-1)} + 1 \leq k \leq S_{d}.
\end{cases}
\end{equation}
To verify \eqref{eq:genDetA}, it suffices to show that 
\begin{equation}\label{eq:genMatBDet}
\det(B)=\left(\prod_{(i_{1},i_{2})\in\mathcal{I}_d}(\lambda_{i_{2}}-\lambda_{i_{1}})^{\omega_{i_{1}}\omega_{i_{2}}}\right)\prod_{(j_{1},j_{2})\in\mathcal{J}}(t_{j_{2}}-t_{j_{1}}).
\end{equation}
For a fixed $d,$ let $\mathcal{C}_{d}$ denote the collection of matrices of the form similar to $B$ for all possible choices of $\omega_{1},$ $\ldots,$ $\omega_{d},$ $\lambda_{1},$ $\ldots,$ $\lambda_{d},$ $t_{1},$ $\ldots,$ $t_{S_{d}}$ satisfying the given conditions of the lemma. We will say $\mathcal{C}_{d}$ is non-singular if the determinant of each matrix in this collection is given by \eqref{eq:genMatBDet} and hence is non-zero. Now let $\mathcal{N}:=\left\{ d\in\mathbb{Z}_{++}:\mathcal{C}_{d}\mbox{ is non-singular}\right\} .$ In these notations, a claim equivalent to \eqref{eq:genMatBDet} is to show that if $d \geq 2,$ then $d \in \mathcal{N}.$ We prove this alternate claim using induction. 

From Lemma \ref{lem:matWNonSingular} it follows that $2 \in \mathcal{N}.$ We treat this as our base case. Let the induction hypothesis be that $d - 1 \in\mathcal{N}$ for some $d \geq 3$. To check if $d \in \mathcal{N},$ we verify whether \eqref{eq:genMatBDet} holds true for the determinant of an arbitrary matrix in $\mathcal{C}_{d}.$ For convenience, we reuse the symbol $B$ to denote this arbitrary matrix. In relation to $B,$ let $g(\lambda_{d}),$ $g^{n}(\lambda_{d}),$ $b_{k}(\lambda_{d}),$ $b_{k}^{r_{k}}(\lambda_{d}),$ $\mathbf{r} \equiv (r_{1},\ldots,r_{\mathcal{S}_{d}})\in\mathbb{Z}_{+}^{S_{d}},$ $\Delta_{S_{d},n}$ and $\mu(\lambda_{d};\mathbf{r})$ be defined as in the proof of Lemma~\ref{lem:matWNonSingular}. 

Our approach is similar in spirit to that used in verifying \eqref{eq:detB}. That is, we treat $\lambda_{d}$ as a variable and the other indeterminates as constants and, using the induction hypothesis, show that
\begin{enumerate}[i.]
\item $g(\lambda_d)$ is a univariate polynomial of degree $S_{d - 1}\omega_d$ and for each $m \in [d - 1],$ $\lambda_m$ is a root with multiplicity exactly $\omega_m \omega_d,$ i.e., $g(\lambda_d) = h \times \prod\limits_{i \in [d - 1]} (\lambda_d - \lambda_i)^{\omega_i\omega_d}.$

\item $h = \left(\prod\limits_{(i_{1},i_{2}) \in \mathcal{I}_{(d - 1)}}(\lambda_{i_{2}}-\lambda_{i_{1}})^{\omega_{i_{1}}\omega_{i_{2}}}\right) \times \left(\prod\limits_{(j_1, j_2) \in \mathcal{J}}(t_{j_2} - t_{j_1})\right).$
\end{enumerate}
To begin with, observe that 
\begin{equation} \label{eq:nthDetgenMat}
g^{n}(\lambda_{d}) = \sum_{\mathbf{r}\in\Delta_{S_{d},n}}\frac{n!}{r_1! \cdots r_{S_d}!}\mu(\lambda_{d};\mathbf{r}).
\end{equation}
Let us now fix an arbitrary $m \in [d - 1].$ Taking hints from the observations made in Lemma~\ref{lem:matWNonSingular}, we define $\mathbf{r}^{*}\in\mathbb{Z}_{+}^{S_{d}}$ by
\[
r_{k}^{*} = 
\begin{cases}
\omega_{d}, & S_{(m-1)} + 1 \leq k \leq S_{m} \vspace{0.15in}\\

0, & \mbox{otherwise}.
\end{cases}
\]
We next partition $\mathbb{Z}_{+}^{S_{d}}$ into the three sets $\mathcal{A}_{\leq,\mathbf{0}},$ $\mathcal{A}_{r^{*}},$ and $\mathcal{A}_{>}$ given respectively by
\begin{multline*}
\mathcal{A}_{\leq,\mathbf{0}} = \left\{\mathbf{r}:0 \leq r_{S_{(m-1)} + 1} < \omega_{d}; \quad \forall 2 \leq k_{1} \leq \omega_{m}, \quad  r_{S_{(m - 1)} + k_{1}} \in \{0\} \cup [\omega_d]\,\right.\\ \left. \mbox{and } r_{k_{2}} = 0 \mbox{ if }S_{(d - 1)} + 1 \leq k_{2} \leq S_{d}\right\},
\end{multline*}
$\mathcal{A}_{\mathbf{r}^{*}}=\left\{ \mathbf{r}^{*}\right\} $ and $\mathcal{A}_{>}=\mathbb{Z}_{+}^{S_{d}}\backslash\left\{ \mathcal{A}_{\mathbf{r}^{*}}\cup\mathcal{A}_{\le,\mathbf{0}}\right\} .$ With respect to $\mathcal{A}_{\leq,\mathbf{0}}$ note that, when $1 \leq k \leq S_{(m-1)}$ or $S_{m} + 1 \leq k \leq S_{(d-1)},$ there is no restriction on $r_{k}.$ This necessarily implies that for each $\mathbf{r} \in \mathcal{A}_{>},$ $\sum_{k}r_{k} > \omega_{d}\omega_{m}.$ Hence these tuples do not appear in the expansion of $g^{n}(\lambda_{d}),$ as given in \eqref{eq:nthDetgenMat}, whenever $n \leq \omega_{m}\omega_{d}.$ So we ignore $\mathcal{A}_{>}$ for the time being and determine the value of $\mu(\lambda_{d};\mathbf{r})$ for tuples lying in the other two sets using the definition of $B$ given in \eqref{eq:BStruct}.
\begin{enumerate}[1.]
\item $\mathbf{r} \in \mathcal{A}_{\leq,\mathbf{0}}:$ We consider the following subcases. 

\begin{enumerate}
\item \noindent $r_{S_{m}} < \omega_{d}:$ Here observe that when $\lambda_{d} = \lambda_{m},$
\[
\left.b_{S_{m}}^{r_{S_{m}}}(\lambda_{d})\right|_{\lambda_{d}=\lambda_{m}} = 
\begin{cases}
\left.b_{S_{d}}(\lambda_{d})\right|_{\lambda_{d}=\lambda_{m}}, & r_{S_{m}}=0 \vspace{0.15in}\\

\begin{array}{c}
\left(\prod\limits_{j=0}^{r_{(S_{m})} - 1}\left(\omega_{d} - j\right)\right) \times \\ 
\left.b_{S_{d} - r_{S_{m}}}(\lambda_{d}) \right|_{\lambda_{d} = \lambda_{m}}
\end{array}, 
& r_{S_{m}} \in [\omega_{d} - 1].
\end{cases}
\]

\item There exists $k \in  [\omega_{m} - 1]$ such that $r_{S_{m}} = r_{S_{m}-1} = \cdots = r_{S_{m}-k+1} = \omega_{d}$ and $r_{S_{m}-k} < \omega_{d}:$
For this subcase, observe that if $\omega_m \leq \omega_d,$ or $\omega_m > \omega_d$ and $1 \leq k < \omega_d,$ then 
\[
\left.b_{S_{m}-k}^{r_{(S_{m}-k)}}(\lambda_{d})\right|_{\lambda_{d}=\lambda_{m}} = 
\begin{cases}
\left.b_{S_{d}-k}(\lambda_{d})\right|_{\lambda_{d}=\lambda_{m}}, & r_{(S_{m} - k)} = 0 \vspace{0.15in}\\

\begin{array}{c}
\left(\prod\limits_{j=0}^{r_{(S_{m}-k)}-1}\left(\omega_{d}-j\right)\right) \times \\

\left.b_{S_{d}-k-r_{(S_{m}-k)}}(\lambda_{d})\right|_{\lambda_{d}=\lambda_{m}}
\end{array},
& 1 \leq r_{(S_{m} - k)} < \omega_{d} - k \vspace{0.15in}\\

\frac{\left.b_{S_{m} + \omega_d - k - r_{S_{m} - k}}^{\omega_d}(\lambda_{d})\right|_{\lambda_{d}=\lambda_{m}}}{\left(\omega_{d}-r_{S_{m}-k}\right)!}, & \omega_{d}-k \leq r_{S_{m} - k} < \omega_{d}
\end{cases}
\]

\noindent On the other hand, if $\omega_m > \omega_d$ and $\omega_d \leq k < \omega_m,$ then
\[
\left.b_{S_{m}-k}^{r_{(S_{m}-k)}}(\lambda_{d})\right|_{\lambda_{d}=\lambda_{m}} = \frac{\left.b_{S_{m} + \omega_d - k - r_{(S_{m} - k)}}^{\omega_d}(\lambda_{d})\right|_{\lambda_{d}=\lambda_{m}}}{\left(\omega_{d}-r_{S_{m}-k}\right)!}, \hspace{0.1in} 0    \leq r_{S_{m}-k} < \omega_{d}.
\]

\end{enumerate}

\noindent From this, it follows that for each $\mathbf{r} \in \mathcal{A}_{\leq,\mathbf{0}}$ there exists a pair of columns which are linearly dependent when $\lambda_{d} = \lambda_{m}$  and thus $\left.\mu(\lambda_{d};\mathbf{r})\right|_{\lambda_{d} = \lambda_{m}} = 0.$

\item $\mathbf{r}\in\mathcal{A}_{\mathbf{r}^{*}}:$ Let $V$ denote the matrix obtained after differentiating element-wise the individual columns of $B$ up to orders as indicated by $\mathbf{r}^{*}$ and substituting $\lambda_{d}=\lambda_{m}.$ It is easy to see that for $j, k  \in [S_d],$ 
\[
\left(V\right)_{jk} = 
\begin{cases}
\begin{array}{c}
\left(\prod\limits_{i \in [d]\backslash \{b,m\}} \left(\lambda_{i} + t_{j}\right)^{\omega_{i}}\right) \times \vspace{0.05in}\\
\left(\lambda_{b}+t_{j}\right)^{k-S_{(b-1)}-1}\left(\lambda_{m}+t_{j}\right)^{\omega_{m}+\omega_{d}}\end{array},
& \begin{array}{c}
	S_{(b-1)} < k \leq S_{b}\\
	\mbox{for some } b\in[d] \backslash \{m,d\}.

\end{array} \vspace{0.15in}\\

\begin{array}{c}
\omega_{d}!\left(\prod\limits_{i\in[d]\backslash\{m,d\}}\left(\lambda_{i}+t_{j}\right)^{\omega_{i}}\right) \times \vspace{0.05in}\\
\left(\lambda_{m}+t_{j}\right)^{k-S_{(m-1)}-1}
\end{array},
& S_{(m-1)} < k \leq S_{m} \vspace{0.15in}\\

\begin{array}{c}
\left(\prod\limits_{i \in [d]\backslash\{m,d\}}\left(\lambda_{i} + t_{j}\right)^{\omega_{i}}\right) \times \vspace{0.05in}\\

\left(\lambda_{m}+t_{j}\right)^{\omega_{m}+k-S_{(d-1)}-1}
\end{array},
& S_{(d-1)} < k \leq S_{d}.
\end{cases}
\]
Now observe that if we define the matrix $\tilde{V}$ from $V$ as 
\[
(\tilde{V})_{jk} = 
\begin{cases} 
\left(V\right)_{jk}, & 1 \leq k \leq S_{(m-1)}, \vspace{0.15in}\\

\frac{1}{\omega_{d}!}(V)_{jk}, & S_{(m-1)} + 1 \leq k \leq S_{m} \vspace{0.15in}\\

(V)_{j(k-S_{m}+S_{(d-1)})}, & S_{m} + 1 \leq k \leq S_{m} + \omega_{d} \vspace{0.15in}\\

(V)_{j(k-\omega_{d})}, & S_{m}+\omega_{d} + 1 \leq k \leq S_{d},
\end{cases}
\]
then $\tilde{V} \in \mathcal{C}_{d - 1}.$ From the induction hypothesis, it follows that
\begin{eqnarray*}
\det{\tilde{V}} & = & (-1)^{S_{(d - 1)} - S_{(m - 1)}} \left(\prod_{j = 1}^{d - 1} (\lambda_m - \lambda_j)^{\omega_j \omega_d} \right) \left( \prod_{(j_1, j_2) \in \mathcal{J}} (t_{j_2} - t_{j_1}) \right) \times \\
& & \left(\prod_{(i_1, i_2) \in \mathcal{I}_{(d - 1)}} (\lambda_{i_2} - \lambda_{i_1})^{\omega_{i_1}\omega_{i_2}}\right).
\end{eqnarray*}
Consequently, 
\begin{eqnarray*}
\mu(\lambda_{d};\mathbf{r}^{*})\bigr|_{\lambda_{d}=\lambda_{m}} & = & \det(V) \\
& = & \left(\prod_{j = 1}^{d - 1} (\lambda_m - \lambda_j)^{\omega_j \omega_d} \right) \left( \prod_{(j_1, j_2) \in \mathcal{J}} (t_{j_2} - t_{j_1}) \right) \times \nonumber\\ 
&   & \left(\prod_{(i_1, i_2) \in \mathcal{I}_{(d - 1)}} (\lambda_{i_2} - \lambda_{i_1})^{\omega_{i_1}\omega_{i_2}}\right). \nonumber
\end{eqnarray*}
\end{enumerate} \vspace{0.15in}

\noindent Now for each $n < \omega_{m}\omega_{d},$ note that $\Delta_{S_{d},n} \subset \mathcal{A}_{\le,0}$ while for $n = \omega_{m}\omega_{d},$ $\Delta_{S_{d},n} \subset \mathcal{A}_{\leq,0} \cup \mathcal{A}_{\mathbf{r}^{*}}.$ From the above observations and \eqref{eq:nthDetgenMat}, it then follows that, for each $0 \leq n < \omega_m \omega_d,$ $g^{n}(\lambda_{d}) \bigr|_{\lambda_{d} = \lambda_{m}} = 0$ while
\begin{eqnarray} \label{eq: gDerLmLd}
g^{\omega_{m}\omega_{d}}(\lambda_{d})\bigr|_{\lambda_{d} = \lambda_{m}} & = & (\omega_m \omega_d)! \left(\prod_{(i_1, i_2) \in \mathcal{I}_{(d - 1)}} (\lambda_{i_2} - \lambda_{i_1})^{\omega_{i_1}\omega_{i_2}}\right) \times \\ 
& &  \left( \prod_{(j_1, j_2) \in \mathcal{J}} (t_{j_2} - t_{j_1}) \right) \left(\prod_{j = 1}^{d - 1} (\lambda_m - \lambda_j)^{\omega_j \omega_d} \right). \nonumber
\end{eqnarray}
Clearly, $g^{\omega_{m}\omega_{d}}(\lambda_{d})\bigr|_{\lambda_{d} = \lambda_{m}} \neq 0.$ Thus, $\lambda_{m}$ is a root of $g(\lambda_{d})$ of multiplicity $\omega_{m}\omega_{d}.$ 

Since $m$ was arbitrary to start with, it follows that for every $i \in [d-1],$ $(\lambda_{d} -\lambda_{i})^{\omega_{i}\omega_{d}}$ is a factor of $g(\lambda_{d}).$ This suggests that original structure of $g(\lambda_{d})$ must have been of the form $g(\lambda_{d}) = h(\lambda_{d})\prod_{i\in[d-1]}\left(\lambda_{d}-\lambda_{i}\right)^{\omega_{i}\omega_{d}}$ for some univariate polynomial $h(\lambda_{d}).$ The fact that the determined multiplicities of the roots of $g(\lambda_{d})$
are exact ensures that $h(\lambda_{i}) \neq 0$ $\forall i \in [d - 1].$ Thus $h$ is not an identically zero polynomial. 

It remains to show that $h$ is a constant and in particular 
\begin{equation} \label{eq:hVal}
h = \left(\prod\limits_{(i_{1},i_{2}) \in \mathcal{I}_{(d - 1)}}(\lambda_{i_{2}}-\lambda_{i_{1}})^{\omega_{i_{1}}\omega_{i_{2}}}\right) \times \left(\prod\limits_{(j_1, j_2) \in \mathcal{J}}(t_{j_2} - t_{j_1})\right).
\end{equation}
Towards this observe that if $n \geq S_{(d-1)}\omega_{d} + 1,$ then for every tuple $\mathbf{r} \in \Delta_{S_{d},n}$ either one of $r_{S_{(d-1)}+1},$ $\ldots,$ $r_{S_{d}}$ is strictly greater than zero or one amongst $r_{1}, \ldots, r_{S_{(d-1)}}$ is strictly bigger than $\omega_{d}.$ In the former situation, arguing
as in observations (1) and (2) from Lemma (\ref{lem:matWNonSingular}), it is easy to see that $\mu(\lambda_{d};\mathbf{r})\equiv 0.$ For the latter situation, let us suppose that for some arbitrary $1 \leq k_{0} \leq S_{(d-1)},$ $r_{k_{0}} > \omega_{d}.$ The fact that every element $(B)_{jk_{0}}$
of column $k_{0}$ is a polynomial in $\lambda_{d}$ of degree $\omega_{d}$ immediately shows that $b_{k_{0}}^{r_{k_{0}}}(\lambda_{d}) \equiv 0$ and hence $\mu(\lambda_{d};\mathbf{r})\equiv0.$ Consequently, it follows that $g^n(\lambda_d) \equiv 0$ $\forall n > S_{(d-1)}\omega_d.$ That is, the degree of $g(\lambda_{d})$ is exactly $S_{(d-1)}\omega_{d}.$ In other
words, 
\begin{equation} \label{eq:eqFormg}
g(\lambda_{d})=h\times\prod_{i\in[d-1]}\left(\lambda_{d}-\lambda_{i}\right)^{\omega_{i}\omega_{d}}
\end{equation}
for some constant $h \neq 0.$ Differentiating \eqref{eq:eqFormg} up to order $\omega_m\omega_d$ for some arbitrary $m \in [d - 1]$ and comparing with \eqref{eq: gDerLmLd}, \eqref{eq:hVal} immediately follows as desired. 
\end{proof}

We now finally prove Lemma~\ref{lem:genTtauMatrix} by showing the following general result.

\begin{theorem} \label{thm:gentTtauMatrix}
For a fixed integer $d \geq 2,$ let $\omega_{1},$ $\ldots,$ $\omega_{d}$ be arbitrary natural numbers. Let $S_{0}=0$ and for $k \in [d],$ $S_{k}=S_{k - 1} + \omega_{k}.$ Further suppose that $\lambda_{1},$ $\ldots,$ $\lambda_{d}$ and $t_{1},$ $\ldots,$ $t_{S_{d}}$ are strictly positive real numbers satisfying $\lambda_{i_{1}}\neq\lambda_{i_{2}},$ if $i_{1}\neq i_{2}$ and $t_{j_{1}} \neq t_{j_{2}},$ if $j_{1} \neq j_{2}.$
Then the square matrix
\begin{equation} \label{eq:genTtauMat}
M :=  \left(\left.
\begin{array}{ccc}
\frac{t_{1}}{(\lambda_{1}+t_{1})} & \ldots & \frac{t_{1}^{\omega_{1}}}{(\lambda_{1}+t_{1})^{\omega_{1}}}\\
\vdots & \ddots & \vdots\\
\frac{t_{S_{d}}}{(\lambda_{1}+t_{S_{d}})} & \ldots & \frac{t_{S_{d}}^{\omega_{1}}}{(\lambda_{1}+t_{S_{d}})^{\omega_{1}}}
\end{array}
\right|
\cdots
\left|
\begin{array}{ccc}
\frac{t_{1}}{(\lambda_{d}+t_{1})} & \ldots & \frac{t_{1}^{\omega_{d}}}{(\lambda_{d}+t_{1})^{\omega_d}}\\
\vdots & \ddots & \vdots\\
\frac{t_{S_{d}}}{(\lambda_{d}+t_{S_{d}})} & \ldots & \frac{t_{S_{d}}^{\omega_{d}}}{(\lambda_{d}+t_{S_{d}})^{\omega_{d}}}
\end{array}\right.\right),
\end{equation}
is non-singular. 
\end{theorem}

\begin{proof} 
For $k \in [S_d],$ let $b_k := \min\{i \in [d]: S_i \geq k\}.$ Now using the expansion $\lambda^{n-1}=\sum_{q=0}^{n-1}(-1)^{q}\tbinom{n-1}{q}t^{q}(\lambda+t)^{n-1-q},$ for $n \in \mathbb{Z}_{++},$ observe that 
\begin{equation} \label{eq:bExp}
\sum_{q=0}^{n-1}(-1)^{q}\tbinom{n-1}{q}\frac{t^{q+1}}{(\lambda+t)^{q+1}} =  \frac{\lambda^{n-1}t}{(\lambda+t)^{n}}.
\end{equation}
Using this note that, for any $j,k \in [S_d],$ 
\begin{equation} \label{eq:ColRel}
\sum_{q = 0}^{k - S_{(b_k - 1)} - 1} (-1)^{q}\tbinom{k - S_{(b_k - 1)} - 1}{q} \left(M\right)_{j(S_{(b_k - 1)} + 1 + q)} = \frac{\lambda_{b_k}^{k - S_{(b_k  - 1)} - 1}t_j}{(\lambda_{b_k} + t_j)^{k - S_{(b_k - 1)}}}.
\end{equation}
From this it follows that if, for each $i \in [d],$ we perform the column operations
\[
\mbox{col}_{k} \leftarrow \sum_{q=0}^{k-S_{(i-1)}-1}(-1)^{q}\tbinom{k-S_{(i-1)}-1}{q}\mbox{col}_{S_{(i-1)}+1+q}
\]
in the order $k = S_{i}, S_{i} - 1, \ldots, S_{(i - 1)} + 1,$ then we end up with matrix
\[
U = \left(\left.
\begin{array}{ccc}
\frac{t_{1}}{(\lambda_{1}+t_{1})} & \ldots & \frac{\lambda_{1}^{\omega_{1}-1}t_{1}}{(\lambda_{1}+t_{1})^{\omega_{1}}}\\
\vdots & \ddots & \vdots\\
\frac{t_{S_{d}}}{(\lambda_{1}+t_{S_{d}})} & \ldots & \frac{\lambda_{1}^{\omega_{1}-1}t_{S_{d}}}{(\lambda_{1}+t_{S_{d}})^{\omega_{1}}}
\end{array}\right|\cdots\left|\begin{array}{ccc}
\frac{t_{1}}{(\lambda_{d}+t_{1})} & \ldots & \frac{\lambda_{d}^{\omega_{d}-1}t_{1}}{(\lambda_{d}+t_{1})^{\omega_{d}}}\\
\vdots & \ddots & \vdots\\
\frac{t_{S_{d}}}{(\lambda_{d}+t_{S_{d}})} & \ldots & \frac{\lambda_{d}^{\omega_{d}-1}t_{S_{d}}}{(\lambda_{d}+t_{S_{d}})^{\omega_{d}}}
\end{array}\right.\right).
\]
Since only reversible column opertions are used to obtain $U$ from $M$, it suffices to show that $U$ is non-singular. But this is true from Lemma~\ref{lem:genMatWNonSingular}. The desired result thus follows.\qquad \end{proof}

\section{Coefficient expansion} \label{sec: Apdx CTexp}

\begin{lemma}
For any ${\boldsymbol \omega} \in \mathbb{Z}_{+}^{d},$ if $\mathcal{D}(\boldsymbol \omega) \neq \emptyset,$ then
\begin{equation} \label{eq:coefExpDup}
\Lambda^{\boldsymbol \omega} = \sum_{k \in \mathcal{D}(\boldsymbol \omega)}\sum_{q = 1}^{\omega_{k}} \gamma_{kq}(\boldsymbol \omega)\Lambda_{k}^{q}(t)
\end{equation}
for all $t.$ Further, this expansion is unique.
\end{lemma}
\begin{proof}
The uniqueness of the expansion is a simple consequence of Lemma~\ref{eq:genTtauMat}. We prove \eqref{eq:coefExpDup} using induction on $\deg(\boldsymbol \omega) = \sum_{k=1}^{d}\omega_{k}.$

For $\deg(\boldsymbol \omega) = 2,$ our basis step, we verify \eqref{eq:coefExpDup} by considering the following exhaustive cases.
\begin{enumerate}[1.]
\item there exists a unique $k \in [d]$ such that $\omega_{k} = 2$: Here, with $\mathcal{D}(\boldsymbol \omega) = \left\{ k\right\},$ observe that $\bar{\Delta}_{k1}(\boldsymbol \omega) = \emptyset$ and $\bar{\Delta}_{k2}(\boldsymbol \omega) = \left\{ \mathbf{0}\in\mathbb{Z}_{+}^{d}\right\}.$ Consequently, $\gamma_{11}(\boldsymbol \omega) = 0$ and $\gamma_{12}(\boldsymbol \omega) = 1.$ Equation~\eqref{eq:coefExpDup} thus trivially holds.

\item there exists unique $k_{1}, k_{2} \in [d]$ such that $\omega_{k_{1}} = \omega_{k_{2}} = 1$: Here observe that $\mathcal{D}(\boldsymbol \omega) = \left\{k_{1}, k_{2}\right\},$ $\bar{\Delta}_{k_{1}1}(\boldsymbol \omega) = \bar{\Delta}_{k_{2}1}(\boldsymbol \omega) = \left\{\mathbf{0} \in \mathbb{Z}_{+}^{d}\right\},$ $\gamma_{k_{1}1}(\boldsymbol \omega) = \beta_{k_{1}k_{2}}$ and $\gamma_{k_{2}1}(\boldsymbol \omega) = \beta_{k_{2}k_{1}}.$ Equation~\eqref{eq:coefExpDup} clearly holds true for this case since $\forall t$
\begin{eqnarray*}
\Lambda^{\boldsymbol \omega} & = & \Lambda_{k_{1}}(t)\Lambda_{k_{2}}(t)\\
& = & \beta_{k_{1}k_{2}}\Lambda_{k_{1}}(t) + \beta_{k_{2}k_{1}}\Lambda_{k_{2}}(t)\\
& = & \gamma_{k_{1}1}(\boldsymbol \omega)\Lambda_{k_{1}}(t) + \gamma_{k_{2}1}(\boldsymbol \omega)\Lambda_{k_{2}}(t).
\end{eqnarray*}
\end{enumerate}

\noindent Now for some fixed $n \geq 2,$ let the strong induction hypothesis be that \eqref{eq:coefExpDup} holds for all ${\boldsymbol \omega}$ such that $\deg(\boldsymbol \omega) \leq n.$ To verify \eqref{eq:coefExpDup} when $\deg(\boldsymbol \omega) = n + 1,$ we consider the following exhaustive cases.
\begin{enumerate}[1.]
\item there exists a unique $k \in [d]$ such that $\omega_{k} = n + 1:$ Here $\mathcal{D}(\boldsymbol \omega) = \{k\},$ 
\[
\bar{\Delta}_{kq}(\boldsymbol \omega) =
\begin{cases}
\emptyset, & q < n + 1\\
\left\{\mathbf{0} \in \mathbb{Z}_{+}^{d}\right\},  & q = n + 1
\end{cases}
\]
and consequently
\[
\gamma_{kq}(\boldsymbol \omega) = 
\begin{cases}
0, & q < n + 1\\
1, & q = n + 1.
\end{cases}
\]
Using these note that \eqref{eq:coefExpDup} again holds trivially. 

\item there exist unique $k_{1}, k_{2} \in [d]$ such that $\omega_{k_{1}}, \omega_{k_{2}} > 0$ and $\omega_{k_{1}} + \omega_{k_{2}} = n + 1$: The fact that $n \geq 2$ immediately ensures that at least one of $\omega_{k_{1}}, \omega_{k_{2}}$ is strictly bigger than unity. Without loss of generality, let us assume that $k_{1} = 1,$ $k_{2} = 2$ and $\omega_{1} > 1.$ Observe then that $\mathcal{D}(\boldsymbol \omega) = \{1, 2\}.$ Also, for each $q \in [\omega_1],$ we have $\bar{\Delta}_{1q}(\boldsymbol \omega) = \left\{ (0,\omega_{1}-q,0,\ldots,0)\in\mathbb{Z}_{+}^{d}\right\},$ a singleton set. Hence $\gamma_{1q}(\boldsymbol \omega) = \beta_{12}^{\omega_{2}}\tbinom{\omega_{2}+\omega_{1}-q-1}{\omega_{2}-1}\beta_{21}^{\omega_{1}-q}.$ Similarly, $\forall q\in[\omega_{2}],$ $\gamma_{2q}(\boldsymbol \omega)=\beta_{21}^{\omega_{1}}\tbinom{\omega_{1}+\omega_{2}-q-1}{\omega_{1}-1}\beta_{12}^{\omega_{2}-q}.$ Observe next that if $\omega_1^{\prime} := \omega_1 - 1,$ a positive number, then $\Lambda^{\boldsymbol \omega}$ can be written as $\Lambda_1(t)\left\{ \Lambda^{\boldsymbol \omega^{\prime}}\right\},$ where $\boldsymbol \omega^{\prime} = (\omega_1^\prime, \omega_2, 0, \ldots, 0).$ Clearly, $\deg(\boldsymbol \omega^{\prime}) = n,$ $\gamma_{1q}(\boldsymbol \omega^{\prime}) = \beta_{12}^{\omega_2} \tbinom{\omega_2 + \omega_1^\prime - q - 1}{\omega_2 - 1} \beta_{21}^{\omega_1^\prime - q},$ $\forall q \in [\omega_1^\prime],$ and $\gamma_{2q}(\boldsymbol \omega^{\prime}) = \beta_{21}^{\omega_1^\prime} \tbinom{\omega_1^\prime + \omega_2 - q - 1}{\omega_1^\prime - 1} \beta_{12}^{\omega_2 - q},$ $\forall q \in [\omega_2].$ By induction hypothesis, it then follows that 
\begin{eqnarray}
\Lambda^{\boldsymbol \omega} & = &\sum_{q=1}^{\omega_{1}^{'}}\gamma_{1q}(\boldsymbol \omega^{\prime})\Lambda_{1}^{q + 1}(t) + \sum_{q = 1}^{\omega_{2}}\gamma_{2q}(\boldsymbol \omega^{\prime})\Lambda_1(t)\Lambda_{2}^{q}(t) \nonumber \\
& := & \mbox{term}_{1}+\mbox{term}_{2}.\label{eq:TermSplit}
\end{eqnarray}
Since $\omega_1^\prime = \omega_1 - 1,$ it follows that 
\begin{eqnarray}
\mbox{term}_{1} & = & \sum_{q=1}^{\omega_{1}^{'}}\beta_{12}^{\omega_{2}}\tbinom{\omega_{2}+\omega_{1}^{'}-q-1}{\omega_{2}-1}\beta_{21}^{\omega_{1}^{'}-q}\Lambda_{1}^{q+1}(t) \nonumber \\ 
& = & \sum_{q=2}^{\omega_{1}}\beta_{12}^{\omega_{2}}\tbinom{\omega_{2}+\omega_{1}-q-1}{\omega_{2}-1}\beta_{21}^{\omega_{1}-q}\Lambda_{1}^{q}(t) \nonumber \\
& = & \sum_{q=2}^{\omega_{1}}\gamma_{1q}(\boldsymbol \omega)\Lambda_{1}^{q}(t).\label{eq:1stFinalExp}
\end{eqnarray}
For $\mbox{term}_2,$ we consider two subcases.
\begin{enumerate}[(a)]
\item $\omega_2 = 1$ or equivalently $\omega_1 = n$: Here, $\mbox{term}_2 = \beta_{21}^{n - 1} \Lambda_1(t) \Lambda_2(t).$ By additionally using the base case, it follows that 
\begin{eqnarray}
\mbox{term}_{2} & = & \beta_{21}^{n-1}\beta_{12}\Lambda_{1}(t)+\beta_{21}^{n}\Lambda_{2}(t) \nonumber \\
 & = & \gamma_{11}(\boldsymbol \omega)\Lambda_{1}(t) + \gamma_{21}(\boldsymbol \omega)\Lambda_{2}(t) \label{eq:2ndFinalExpSCa}
\end{eqnarray}
Substituting \eqref{eq:1stFinalExp} and \eqref{eq:2ndFinalExpSCa} in \eqref{eq:TermSplit}, it is clear that \eqref{eq:coefExpDup} holds.

\item $\omega_2 > 1:$ We again apply the induction hypothesis individually to the term $\Lambda_{1}(t)\Lambda_{2}^{q}(t)$ for each $1 \leq q \leq \omega_{2}.$ Note that none of these expansions yield scalar multiples of $\Lambda_{1}^{q}(t)$ for any $2 \leq q \leq \omega_{1}$ and, thus, their associated weights  obtained in \eqref{eq:1stFinalExp} remain unchanged in the overall decomposition of $\Lambda^{\boldsymbol \omega}.$ The term $\Lambda_{1}(t)$ is, however, definitely present in the final decomposition of $\mbox{term}_{2}.$ For each $q \in [\omega_2],$ observe that the constant associated with $\Lambda_1(t)$ in the expansion of $\Lambda_1(t)\Lambda_2^q(t)$ is $\gamma_{11}(1, q, 0, \ldots, 0) = \beta_{12}^{q}.$ From the definition of $\mbox{term}_2$, clearly, 
the constant associated to $\Lambda_1(t)$ in the final decomposition of $\Lambda^{\boldsymbol \omega}$ is
\begin{eqnarray*}
\sum_{q = 1}^{\omega_2} \gamma_{2q}(\boldsymbol \omega^\prime)\beta_{12}^q 
& = & \beta_{21}^{\omega_{1}^{'}}\beta_{12}^{\omega_{2}}\sum_{q=1}^{\omega_{2}}\tbinom{\omega_{1}^{'}+\omega_{2}-q-1}{\omega_{1}^{'}-1}\\
& = &\beta_{21}^{\omega_{1}^{'}}\beta_{12}^{\omega_{2}}\tbinom{\omega_{1}^{'}+\omega_{2}-1}{\omega_{1}^{'}},\\
& = &\beta_{21}^{\omega_{1}-1}\beta_{12}^{\omega_{2}}\tbinom{\omega_{2}+\omega_{1}-1-1}{\omega_{2}-1}\\
& = &\gamma_{11}(\boldsymbol \omega).
\end{eqnarray*}
Note that the second equality above  is a consequence of the well known binomial identity $\sum_{r=j}^{n}\tbinom{r}{j}=\tbinom{n+1}{j+1}.$ Next observe that the decomposition of $\Lambda_{1}(t)\Lambda_{2}^{q}(t)$ by the induction hypothesis also results in a linear combination of $\Lambda_{2}(t),$ $\Lambda_{2}^{2}(t),$ $\ldots,$ $\Lambda_{2}^{q}(t).$ This implies that 
\begin{equation} \label{eq:intterm2Exp}
\mbox{term}_{2} = \gamma_{11}(\boldsymbol \omega)\Lambda_{1}(t)+ \sum_{q=1}^{\omega_{2}}v_{q}\Lambda_{2}^{q}(t)
\end{equation}
for some $t-$independent constant $\{v_q\}.$ By substituing \eqref{eq:1stFinalExp} and \eqref{eq:intterm2Exp} in \eqref{eq:TermSplit}, it follows that
\begin{eqnarray}\label{eq:SplitLbda1}
\Lambda^{\boldsymbol \omega} & = & \sum_{q=1}^{\omega_{1}}\gamma_{1q}(\boldsymbol \omega)\Lambda_{1}^{q}(t)+\sum_{q=1}^{\omega_{2}}v_{q}\Lambda_{2}^{q}(t).
\end{eqnarray}

Now recall that $\omega_{2} > 1$ in the subcase under consideration. Expressing $\Lambda^{\boldsymbol \omega}$ as $\Lambda_{2}(t)\left\{\Lambda_{1}^{\omega_{1}}(t)\Lambda_{2}^{\omega_{2} - 1}(t)\right\}$ and interchanging the roles of $\Lambda_1(t)$ and $\Lambda_2(t)$ in the above argument, it then follows that 
\begin{eqnarray}\label{eq:SplitLbda2}
\Lambda^{\boldsymbol \omega} & = & \sum_{q=1}^{\omega_{1}}u_{q}\Lambda_{1}^{q}(t) + \sum_{q=1}^{\omega_{2}}\gamma_{2q}(\boldsymbol \omega)\Lambda_{2}^{q}(t)
\end{eqnarray}
for some $t-$independent constants $\{u_q\}.$ Comparing \eqref{eq:SplitLbda1} and \eqref{eq:SplitLbda2}, we get
\begin{equation} \label{eq:EqCoef}
\sum_{q=1}^{\omega_{1}}\left\{\gamma_{1q}(\boldsymbol \omega) - u_{q}\right\} \Lambda_{1}^{q}(t) + \sum_{q=1}^{\omega_{2}}\left\{v_{q}-\gamma_{2q}(\boldsymbol \omega)\right\} \Lambda_{2}^{q}(t)=0
\end{equation}
for all $t.$ A simple application of Lemma~\ref{eq:genTtauMat} then shows that $u_q = \gamma_{1q}(\boldsymbol \omega),$ $\forall q \in [\omega_1],$ and $v_q = \gamma_{2q}(\boldsymbol \omega),$ $\forall q \in [\omega_2].$ By substituting these values in either \eqref{eq:SplitLbda1} or \eqref{eq:SplitLbda2}, it follows that \eqref{eq:coefExpDup} holds for this case also. 
\end{enumerate}

\item Cardinality of the set $\mathcal{D}(\boldsymbol \omega) = \left\{ k\in[d]:\omega_{k}>0\right\}$ is atleast $3:$ Without loss of generality, let us suppose that $\omega_{1}, \omega_{2} > 0.$ Our approach here is to express $\Lambda^{\boldsymbol \omega}$ as $\Lambda_{1}^{\omega_{1}}(t)\left\{ \Lambda_{2}^{\omega_{2}}(t)\ldots\Lambda_{d}^{\omega_{d}}(t)\right\}$ and expand the term within the braces using the induction hypothesis. With $\boldsymbol \omega^{\prime\prime} = (0,\omega_{2},\ldots,\omega_{d}),$ it follows that 
\begin{eqnarray*}
\Lambda^{\boldsymbol \omega} & = & \Lambda_{1}^{\omega_{1}}(t)\left\{ \sum_{k \in \mathcal{D}(\boldsymbol \omega)\backslash \left\{1\right\}} \sum_{q = 1}^{\omega_{k}}\gamma_{kq}(\boldsymbol \omega^{\prime\prime})\Lambda_{k}^{q}(t)\right\} \\
 & = & \sum_{k\in\mathcal{D}(\boldsymbol \omega)\backslash\left\{ 1\right\} }\sum_{q=1}^{\omega_{k}}\gamma_{kq}(\boldsymbol \omega^{\prime\prime})\Lambda_{1}^{\omega_{1}}(t)\Lambda_{k}^{q}(t)\\
 & := & \sum_{k\in\mathcal{D}({\boldsymbol \omega})\backslash\left\{ 1\right\} }\mbox{term}_{k}.
\end{eqnarray*}
A neat observation, by virtue of our induction hypothesis, is that the eventual constant associated with $\Lambda_{k}^{q}(t)$, $k \in \mathcal{D} ({\boldsymbol \omega})\backslash\left\{ 1\right\} ,$ $q \in [\omega_{k}],$ depends solely on $\mbox{term}_{k}.$ Infact the only terms that contribute a scalar multiple of $\Lambda_{k}^{q}(t)$ are
$\Lambda_{1}^{\omega_{1}}(t)\Lambda_{k}^{q}(t),$ $\Lambda_{1}^{\omega_{1}}(t)\Lambda_{k}^{q+1}(t), \ldots, \Lambda_{1}^{\omega_{1}}(t)\Lambda_{k}^{\omega_{k}}(t).$ With this
in mind, we focus on $\mbox{term}_{2}$ and evaluate the eventual constant, say $\alpha,$ associated with $\Lambda_{2}^{q}(t)$ for some arbitrary $q \in [\omega_{2}].$ Firstly observe that the scalar multiple of $\Lambda_{2}^{q}(t)$ in the expansion of $\Lambda_{1}^{\omega_{1}}(t)\Lambda_{2}^{q+j}(t),$ when $0 \leq j \leq \omega_{2}-q,$ is $\gamma_{2q}(\omega_{1},q+j,0,\ldots,0).$ This implies that
\begin{eqnarray*}
\alpha & = & \sum_{j=0}^{\omega_{2}-q}\gamma_{2(q+j)}(\boldsymbol \omega^{\prime\prime})\gamma_{2q}(\omega_{1},q+j,0,\ldots,0)\\
 & := & \sum_{j=0}^{\omega_{2}-q}c_{j}.
\end{eqnarray*}
In relation to each $c_{j},$ observe that $\bar{\Delta}_{2q}(\omega_{1},q+j,0,\ldots,0)=\left\{ \left(j,0,\ldots,0\right)\right\},$ a singleton set, and for any $\mathbf{s}^{'} \equiv (s_{1}^{'},\ldots,s_{d}^{'}) \in \bar{\Delta}_{2(q+j)}(\boldsymbol \omega^{\prime\prime}),$ $s_{1}^{'} = 0$ and $\sum_{r=2}^{d}s_{r}^{'}=\omega_{2}-q-j.$ This implies that $\mathbf{s}^{'} \in \bar{\Delta}_{2(q+j)}(\boldsymbol \omega^{\prime \prime})$ if and only if the tuple $\left(j,s_{2}^{'},s_{3}^{'},\ldots,s_{d}^{'}\right) \in \bar{\Delta}_{2q}(\boldsymbol \omega)\bigcap\mathcal{B}_{j},$ where $\mathcal{B}_{j}:=\left\{ \mathbf{s}\in\mathbb{Z}_{+}^{d}:\, s_{1}=j\right\}.$ Thus, 
\[
c_{j} = \left(\prod_{r \in \mathcal{D}(\boldsymbol \omega)}\beta_{2r}^{\omega_{r}}\right)\sum_{\mathbf{s}\in\bar{\Delta}_{2q}(\boldsymbol \omega) \bigcap \mathcal{B}_{j}} \prod_{r\in \mathcal{D}(\boldsymbol \omega)}\tbinom{\omega_{r}+s_{r}-1}{\omega_{r}-1}\beta_{r2}^{s_{r}}.
\]

Consequently, it follows that $\alpha = \gamma_{2q}(\boldsymbol \omega).$ By replicating this argument for every $k\in\mathcal{D}\backslash\left\{ 1\right\} $
and every $d\in[\omega_{k}],$ observe that the weight associated with $\Lambda_{k}^{q}(t)$ in the final expansion is $\gamma_{kq}(\boldsymbol \omega).$ This, combined with the fact that the decomposition of $\Lambda_{1}^{\omega_{1}}(t)\Lambda_{k}^{q}(t)$ also results in a linear combination of $\Lambda_{1}(t),$ $\ldots,$ $\Lambda_{1}^{\omega_{1}}(t),$ shows that 
\[
\Lambda^{\boldsymbol \omega} = \sum_{q=1}^{\omega_{1}}u_{q}\Lambda_{1}^{q}(t)+\sum_{k\in\mathcal{D}(\boldsymbol \omega)\backslash\left\{ 1\right\} }\sum_{q=1}^{\omega_{k}}\gamma_{kq}(\boldsymbol \omega)\Lambda_{k}^{q}(t)
\]
for some $t-$independent constants $u_{1},$ $\ldots,$ $u_{\omega_{1}}.$ But observe that $\Lambda^{\boldsymbol \omega}$ can also be expressed as $\Lambda_{2}^{\omega_{2}}(t)\left\{ \Lambda_{1}^{\omega_{1}}(t)\Lambda_{3}^{\omega_{3}}(t)\ldots\Lambda_{d}^{\omega_{d}}(t)\right\}.$ By repeating the above arguments with an interchange of the roles of $\Lambda_{1}(t)$ and $\Lambda_{2}(t),$ we, thus, also get
\[
\Lambda^{\boldsymbol \omega}=\sum_{q=1}^{\omega_{2}}v_{q}\Lambda_{2}^{q}(t)+\sum_{k\in\mathcal{D}(\boldsymbol \omega)\backslash\left\{ 2\right\} }\sum_{q=1}^{\omega_{k}}\gamma_{kq}(\boldsymbol \omega)\Lambda_{k}^{q}(t)
\]
for some real constants $v_{1},$ $\ldots,$ $v_{\omega_{2}}.$ Arguing as in case $(2b)$ above, it is easy to see that $u_{q} = \gamma_{1q}(\boldsymbol \omega),$ $\forall q \in [\omega_1]$ and $v_{q} = \gamma_{2q}(\boldsymbol \omega),$ $\forall q \in [\omega_2].$
\end{enumerate}

\noindent This completes the induction argument and, hence, proves the desired result. \qquad
\end{proof}

\section{Regularity of EPS} \label{sec: Apdx RPEPS}
For discussions pertaining to this section, we suppose that the domain and the co-domain of the map $\mathcal{E},$ used to define the EPS of \eqref{eq:genElePolSys}, is $\mathbb{C}^{d \cdot N_i}.$ We now prove Lemma~\ref{lem:EPSWellBeh} through the following series of results.  

\begin{lemma} \label{lem:egRegPt}
Suppose $a_1, \ldots, a_{N_i}$ are $N_i$ distinct complex numbers. Let $\mathbf{a} \equiv (\mathbf{a}_1, \ldots, \mathbf{a}_{N_i}),$ where $\mathbf{a}_j \equiv (a_j, \ldots, a_j) \in \mathbb{C}^{d}.$ Then 
\begin{equation} \label{eq:detEPSNZ}
\det(\dot{\mathcal{E}}(\mathbf{x}))\bigr|_{\mathbf{x} = \mathbf{a}} = \prod_{1 \leq j_1 < j_2 \leq N_i} (a_{j_1} - a_{j_2})^d \neq 0.
\end{equation}
\end{lemma}
\begin{proof}
Pick an arbitrary $\boldsymbol \Omega \in \Delta_{d + 1, N_i},$ $j_2 \in [N_i]$ and $k_2 \in [d].$ We first show that
\begin{equation} \label{eq:DerivativePolg}
\left.\frac{\partial g(\mathbf{x};\boldsymbol \Omega)}{\partial x_{j_{2}k_{2}}}\right|_{\mathbf{x} = \mathbf{a}} =
\begin{cases}
0, & \omega_{k_{2}} = 0 \\
\nu(\boldsymbol \Omega,k_{2}) e_{\mathrm{\eta}(\boldsymbol \omega)}(a_{1},\ldots,a_{j_{2}-1},a_{j_{2}+1},\ldots,a_{N_{i}}), & \mbox{otherwise},
\end{cases}
\end{equation}
where $\nu(\boldsymbol \Omega;k_{2})$ is a number, independent of choice of the vector $\mathbf{a}$ and $j_{2}\in[N_{i}],$ $\eta(\boldsymbol \omega)=\left(\sum_{k=1}^{d}\omega_{k}\right) - 1 \in \mathbb{Z}_{+}$ and $e_{l}(\cdot),$ for $0 \leq l \leq N_i - 1,$ is the $l^{th}$ elementary symmetric polynomial in $N_{i} - 1$ variables.

Suppose that $\omega_{k_2} = 0.$ Then clearly, $\forall \mathbf{b} \in \mathcal{B}_{\boldsymbol \Omega},$ $\theta_{k_2}(\mathbf{b}) = 0.$ This implies that $g(\mathbf{x}; \boldsymbol \Omega)$ is a constant with respect to $x_{j_2k_2}.$ Thus, $\frac{\partial g(\mathbf{x}; \boldsymbol \Omega)}{\partial x_{j_2k_2}}\bigr|_{\mathbf{x} = \mathbf{a}} = 0$ as desired.

Now suppose that $\omega_{k_2} \in [N_i].$ Clearly, 
\[
g(\mathbf{x};\boldsymbol \Omega) =  x_{j_{2} k_{2}} \sum_{\mathbf{b} \in \mathcal{B}_{\boldsymbol \Omega}, \, b_{j_{2}} = k_{2}}\left(\prod_{j=1, \, j \neq j_{2}, \, b_{j} \neq d+1 }^{N_{i}} x_{jb_j} \right) + 
\sum_{\mathbf{b} \in \mathcal{B}_{\boldsymbol \Omega}, \, b_{j_{2}} \neq k_{2}} \left(\prod_{j=1, \, b_{j} \neq d+1}^{N_{i}}x_{jb_{j}}\right).
\]
Hence,
\begin{equation} \label{eq:diffg1}
\left.\frac{\partial g(\mathbf{x};\boldsymbol \Omega)}{\partial x_{j_{2}k_{2}}} \right|_{\mathbf{x} = \mathbf{a}}  = 
\sum_{\mathbf{b}\in\mathcal{B}_{\boldsymbol \Omega}, \, b_{j_2} = k_2}\left(\prod_{j = 1, \, j \neq j_{2}, \, b_{j} \neq d+1}^{N_{i}}a_{j}\right).
\end{equation}
Now for any $\mathcal{S} \subseteq [N_{i}],$ such that $|\mathcal{S}| = \omega_{d+1},$ define the set $\mathcal{B}_{\boldsymbol \Omega}(j_{2},k_{2}, \mathcal{S}) \subseteq \mathcal{B}_{\boldsymbol \Omega}$ by 
\[
\mathcal{B}_{\boldsymbol \Omega}(j_{2},k_{2}, \mathcal{S}) = \left\{ \mathbf{b} \in [d+1]^{N_{i}} : b_{j_2} = k_{2}; \, \forall j \in \mathcal{S}, \, b_{j} = d + 1; \, \Theta(\mathbf{b}) = \boldsymbol \Omega\right\}.
\]
Its cardinality, given by 
\begin{equation} \label{eq:setCard}
|\mathcal{B}_{\boldsymbol \Omega}(j_{2}, k_{2},\mathcal{S})| = 
\tbinom{N_{i} - 1 - \omega_{d+1}}{\omega_{1}, \; \cdots, \; \omega_{(k_{2} - 1)}, \; \omega_{k_2} - 1, \; \omega_{(k_2 + 1)}, \; \cdots, \; \omega_{d}, \; 0},
\end{equation}
clearly depends on $|\mathcal{S}|,$ but not on $\mathcal{S}$ itself. Also, it is independent of choice of the vector $\mathbf{a}$ and $j_2 \in [N_i].$ By defining 
\begin{equation} \label{eq:setCardIndofJ_2}
\nu(\boldsymbol \Omega; k_{2}) = |\mathcal{B}_{\boldsymbol \Omega}(j_{2},k_{2},\mathcal{S})|
\end{equation}
and using the fact that $\left\{\mathbf{b} \in \mathcal{B}_{\boldsymbol \Omega} : b_{j_{2}} = k_{2}\right\} = \bigcup_{\mathcal{S} \subseteq [N_{i}], \, |\mathcal{S}| = \omega_{d+1}} \mathcal{B}_{\boldsymbol \Omega}(j_{2}, k_{2}, \mathcal{S}),$ a disjoint union, observe from \eqref{eq:diffg1} that 
\begin{eqnarray*}
\frac{\partial g(\mathbf{x}; \boldsymbol \Omega)}{\partial x_{j_{2}k_{2}}} & = & 
\sum_{\mathcal{S} \subseteq [N_{i}], \, |\mathcal{S}| = \omega_{d+1}} \nu(\boldsymbol \Omega; k_{2})\left(\prod_{j=1, \, j \notin \mathcal{S} \cup\{j_{2}\}}^{N_{i}} a_{j}\right) \\\vspace{0.2in}
& = & \nu(\boldsymbol \Omega;k_{2})e_{\eta(\boldsymbol \omega)}(a_{1}, \ldots, a_{(j_2 - 1)}, a_{(j_2 + 1)}, \ldots,a_{N_{i}}),
\end{eqnarray*}
as desired. This now completes the verification of \eqref{eq:DerivativePolg}. 

Now fix an arbitrary $k_{1},k_{2} \in [d]$ and $q_{1}, j_{2} \in [N_{i}].$ As of consequence of \eqref{eq:DerivativePolg}, observe that
\begin{equation} \label{eq:hkqGenParDer}
\left.\frac{\partial h_{k_{1}q_{1}}(\mathbf{x})}{\partial x_{j_{2}k_{2}}} \right|_{\mathbf{x} = \mathbf{a}} =
\sum_{
\begin{array}{c}
\boldsymbol \Omega \in \Delta_{d+1,N_{i}},\\
\omega_{k_{1}} \geq q_{1},\\
\omega_{k_{2}} > 0
\end{array}
}
\left[
\begin{array}{c}
\frac{\gamma_{k_1q_1}(\boldsymbol \omega) \nu(\boldsymbol \Omega;k_{2})}{\lambda_{d+1}^{N_{i} - q_{1} - \omega_{d+1}}} \times \\
e_{\eta(\boldsymbol \omega)}(a_{1},\ldots,a_{j_{2}-1},a_{j_{2}+1},\ldots,a_{N_{i}})
\end{array}
\right].
\end{equation}
By grouping the symmetric polynomials in the above equation that have the same degree, i.e., identical values of $\eta(\boldsymbol \omega),$ we get
\begin{equation} \label{eq:diffhjqDefn}
\left.\frac{\partial h_{k_{1}q_{1}}(\mathbf{x})}{\partial x_{j_{2}k_{2}}} \right|_{\mathbf{x} = \mathbf{a}} =
\sum_{l = q_1 - 1}^{N_i - 1} c_l(k_1, q_1, k_2) e_l(a_{1},\ldots,a_{j_{2}-1},a_{j_{2}+1},\ldots,a_{N_{i}})),
\end{equation}
where
\[
c_{l}(k_{1},q_{1},k_{2}) := \sum_{\boldsymbol \Omega \in\Delta_{d + 1,N_{i}}, \, \omega_{k_{1}} \geq q_{1},\, \omega_{k_2} > 0, \,\eta(\boldsymbol \omega) = l}\left(\frac{\gamma_{k_{1}q_{1}}(\boldsymbol \Omega)}{\lambda_{d+1}^{N_{i} - q_1 - \omega_{d + 1}}}\right)\nu(\boldsymbol \Omega;k_{2}).
\]
Using \eqref{eq:GammakqDefn}, \eqref{eq:setCard} and \eqref{eq:setCardIndofJ_2},  note in particular that 
\begin{equation} \label{eq:splCVal}
c_{q_1 - 1}(k_1, q_1, k_2) = 
\begin{cases}
1, & k_1 = k_2\\
0, & \mbox{otherwise}.
\end{cases}
\end{equation}

For every $k_{1},k_{2}\in[N_{i}],$ with
\begin{eqnarray} \label{eq:Xik_1k_2}
\Xi_{k_{1},k_{2}} & := & 
\left(
\begin{array}{ccc}
\frac{\partial h_{k_{1}1}(\mathbf{x})}{\partial x_{1k_{2}}} & \cdots & \frac{\partial h_{k_{1}1}(\mathbf{x})}{\partial x_{N_{i}k_{2}}}\\
\vdots & \ddots & \vdots\\
\frac{\partial h_{k_{1}N_{i}}(\mathbf{x})}{\partial x_{1k_{2}}} & \cdots & \frac{\partial h_{k_{1}N_{i}}(\mathbf{x})}{\partial x_{N_{i}k_{2}}}
\end{array}
\right),
\end{eqnarray}
consider the matrix 
\begin{equation} \label{eq:XiMat}
\Xi:= \dot{\mathcal{E}}(\mathbf{x}) =
\left(
\begin{array}{ccc}
\Xi_{1,1} & \cdots & \Xi_{1,d}\\
\vdots & \ddots & \vdots\\
\Xi_{d,1} & \cdots & \Xi_{d,d}
\end{array}\right).
\end{equation}
Using elementary row operations, our strategy is show that 
\begin{equation} \label{eq:diagFormDetEMat}
\det(\dot{\mathcal{E}}(\mathbf{x})) = 
\det\left(
\begin{array}{cccc}
\Gamma & \mathbf{0} & \cdots & \mathbf{0}\\
\mathbf{0} & \Gamma & \cdots & \mathbf{0}\\
\vdots & \vdots & \ddots & \vdots\\
\mathbf{0} & \mathbf{0} & \cdots & \Gamma
\end{array}
\right) = \left(\det(\Gamma)\right)^{d},
\end{equation}
where 
\begin{equation} \label{eq:GammaMat}
\Gamma=\left(\begin{array}{ccc}
e_{0}(a_{2},\ldots,a_{N_{i}}) & \cdots & e_{0}(a_{1},\ldots,a_{N_{i}-1})\\
\vdots & \ddots & \vdots\\
e_{N_{i}-1}(a_{2},\ldots,a_{N_{i})} & \cdots & e_{N_{i}-1}(a_{1},\ldots,a_{N_{i}-1})
\end{array}\right).
\end{equation}
This is clearly sufficient to verify \eqref{eq:detEPSNZ} since $\det(\Gamma) = \prod_{1 \leq j_{1} < j_{2} \leq N_{i}}(a_{j_{1}}-a_{j_{2}})$ is a well known result. 

For notational convenience, the matrix obtained after applying elementary operations on $\Xi$ will again be referred to by $\Xi.$ The notation $\mbox{row}(r, k_1)$ will mean the $r^{th}$ row of the $k_{1}^{th}$ block matrix row of $\Xi,$ i.e. the $r^{th}$ row of 
\[
[
\begin{array}{ccc}
\Xi_{k_{1},1} & \cdots & \Xi_{k_{1},d}
\end{array}].
\] 
Similarly $\mbox{col}(c, k_2)$ will refer to the $c^{th}$ column of the $k_2^{th}$ block matrix column of $\Xi.$ We will say that the matrix $\Xi$ is in state $n,$ $n\in[N_{i}],$ if:
\begin{enumerate}[1.]
\item $\det(\Xi)=\det(\dot{\mathcal{E}}(\mathbf{x})).$

\item for each $k_1 \in [d]$ and $q \in [N_{i}],$ the $q^{th}$ row of $(\Xi_{k_1, k_2} - \Gamma),$ where $\Gamma$ is as defined in \eqref{eq:GammaMat}, is 
\[
[
\begin{array}{ccc}
\sum\limits_{l=q}^{N_{i}-n}c_{l}(k_{1},q,k_{1})e_{l}(a_{2},\ldots,a_{N_{i}}) & \cdots & \sum\limits_{l=q}^{N_{i}-n}c_{l}(k_{1},q,k_{1})e_{l}(a_{1},\ldots,a_{N_{i}-1})
\end{array}
].
\]

\item for each $k_1, k_2 \in [d],$ such that $k_{1} \neq k_{2},$ and each $q \in [N_{i}],$ the $q^{th}$ row of $\Xi_{k_{1},k_{2}}$ is
\[
[
\begin{array}{ccc}
\sum\limits_{l=q}^{N_{i}-n}c_{l}(k_{1},q,k_{2})e_{l}(a_{2},\ldots,a_{N_{i}}) & \cdots & \sum\limits_{l=q}^{N_{i}-n}c_{l}(k_{1},q,k_{2})e_{l}(a_{1},\ldots,a_{N_{i}-1})
\end{array}
].
\]
\end{enumerate}

From \eqref{eq:diffhjqDefn} and \eqref{eq:splCVal}, it is clear that $\Xi,$ as given in \eqref{eq:XiMat}, is in state $1.$ To verify \eqref{eq:diagFormDetEMat}, it suffices to show that the state of $\Xi$ can be changed to $N_i$ using reversible elementary row operations. We do so by describing the reversible row operations needed to put $\Xi$ in state $n + 1$ starting from a state $n,$ where $1 \leq n < N_i.$

Suppose that $k_1 = k_2.$ Then from the definition of $\Xi$ being in state $n,$ it follows that when $N_i - n < q \leq N_i,$ the $q^{th}$ row of $\Xi_{k_1k_2}$ is given by
\[
[
\begin{array}{ccc}
e_{q-1}(a_{2},\ldots,a_{N_{i}}) & \cdots & e_{q-1}(a_{1},\ldots,a_{N_{i}-1})
\end{array}
],
\]
while, for $1 \leq q \leq N_{i} - n,$ it is given by
\[
[
\begin{array}{ccc}
\sum\limits_{l=q-1}^{N_{i}-n}c_{l}(k_{1},q,k_{1})e_{l}(a_{2},\ldots,a_{N_{i}}) & \cdots & \sum\limits_{l=q-1}^{N_{i}-n}c_{l}(k_{1},q,k_{1})e_{l}(a_{1},\ldots,a_{N_{i}-1})
\end{array}
].
\]
On the other hand, when $k_1 \neq k_2,$ the $q^{th}$ row, $N_i - n < q \leq N_i,$ of $\Xi_{k_1k_2}$ is $\mathbf{0} \in \mathbb{C}^{N_i},$ while for $1 \leq q \leq N_i - n,$ it is given by
\[
[
\begin{array}{ccc}
\sum\limits_{l=q}^{N_{i}-n}c_{l}(k_{1},q,k_{1})e_{l}(a_{2},\ldots,a_{N_{i}}) & \cdots & \sum\limits_{l=q}^{N_{i}-n}c_{l}(k_{1},q,k_{1})e_{l}(a_{1},\ldots,a_{N_{i}-1})
\end{array}
].
\]
From these observations, it is clear that the effect of subtracting a scalar multiple of $\mbox{row}(N_i - n + 1, k),$ for some $k \in [d],$ from any other arbitrary row of $\Xi,$ say $\mbox{row}(r, k_1),$ where $k_1 \in [d]$ and $r \in [N_i],$ is only on the elements whose column index is one of $\mbox{col}(1,k), \ldots, \mbox{col}(N_i, k).$

Now fix an arbitrary $k \in [d]$ and consider the following row operations: 
\begin{enumerate}[1.]
\item $\mbox{row}(N_{i} - n, k) \leftarrow \mbox{row}(N_{i} - n, k) - c_{N_{i} - n}(k, N_{i} - n, k) \times \mbox{row}(N_i - n + 1, k).$ Leaving all other elements unchanged, this operation changes the $(N_i - n)^{th}$ row of the matrix $\Xi_{k, k}$ to
\[
[
\begin{array}{ccc}
e_{N_i - n - 1}(a_{2}, \ldots, a_{N_i}) & \cdots & e_{N_i - n - 1}(a_{1}, \ldots, a_{N_i - 1})
\end{array}
].
\]

\item \noindent $\mbox{row}(q, k) \leftarrow \mbox{row}(q, k) - c_{N_i - n}(k, q, k) \times \mbox{row}(N_i - n + 1, k)$ for each $1 \leq q < N_i - n.$ For every $q,$ leaving other elements unchanged, this operation changes the $q^{th}$ row of $\Xi_{k, k}$ to 
\[
[
\begin{array}{ccc}
\sum\limits_{l = q - 1}^{N_i - (n + 1)}c_{l}(k, q, k)e_{l}(a_{2},\ldots,a_{N_i}) & \cdots & \sum\limits_{l = q - 1}^{N_i - (n + 1)}c_{l}(k, q, k)e_{l}(a_{1},\ldots,a_{N_{i}-1})
\end{array}
].
\]

\item $\mbox{row}(N_i - n, k_1) \leftarrow \mbox{row}(N_{i} - n, k_1) - c_{N_i - n}(k_1, N_i - n, k) \times \mbox{row}(N_i - n + 1, k)$ for each $k_1 \in [d]\backslash k.$ For every $k_1,$ leaving other elements unchanged, this operation changes the $(N_i - n)^{th}$ row of the matrix $\Xi_{k_1, k}$ to $\mathbf{0} \in \mathbb{C}^{N_i}.$

\item $\mbox{row}(q, k_1) \leftarrow \mbox{row}(q, k_1) - c_{N_i - 1}(k_1, q, k) \times \mbox{row}(N_i - n + 1, k)$ for each $k_1\in[d] \backslash k$ and each $1 \leq q <N_i - n.$ For every $k_1, q,$ leaving other elements unchanged, this operation changes the $q^{th}$ row of $\Xi_{k_1, k}$ to 
\[
[
\begin{array}{ccc}
\sum\limits_{l = q}^{N_i - (n + 1)}c_{l}(k_1, q, k)e_{l}(a_{2},\ldots,a_{N_i}) & \cdots & \sum\limits_{l = q}^{N_i - (n + 1)}c_{l}(k_1, q, k)e_{l}(a_{1},\ldots,a_{N_i - 1})\end{array}
].
\]
\end{enumerate}
At present, the matrices $\Xi_{1, k},\ldots,\Xi_{d, k}$ are in the form as required when the $\Xi$ is in state $n + 1.$ By repeating the above operations for each $k \in [d],$ the entire $\Xi$ matrix can thus be put in state $n + 1$ as desired. \qquad \end{proof}

We now recall some standard results from numerical algebraic geometry. Lemmas~\ref{lem:NonEmpDense} and \ref{lem:NonEmZarOpenCandR} may be found in \cite{Mumford99}, Theorems~\ref{thm:Bezout's-Theorem} and \ref{thm:ConstFinSolutions} may be found in \cite{Sommese05} while Theorem~\ref{thm:RegValZarOpen} may be found in \cite{Hartshorne77} or \cite{Li96}. 

\begin{definition}
A set $V \subseteq \mathbb{C}^n$ is said to be Zariski closed if there exist $f_{1}, \ldots, f_{k} \in \mathbb{C}[x_{1},\ldots,x_{n}],$ such that
\[
V = \{\mathbf{x} \equiv(x_{1},\ldots,x_{n}) \in \mathbb{C}^{n}:f_{1}(\mathbf{x})=\cdots=f_{k}(\mathbf{x})=0\}.
\]
\end{definition}
\begin{definition}
A set $V \subseteq \mathbb{C}^n$ is said to be Zariski open if its compliment is Zariski closed.
\end{definition}

By replacing $\mathbb{C}$ with $\mathbb{R}$ everywhere in the above definitions, Zariski closed and Zariski open sets of $\mathbb{R}^n$ can be equivalently defined.

To distinguish from the open and closed sets of the usual Euclidean topology, the open and closed sets of the Zariski topology with always have the word \emph{Zariski} prefixed to them.

\begin{lemma} \label{lem:NonEmpDense} 
A Zariski open subset $\mathcal{O}$ of $\mathbb{R}^{n}$ $(\mbox{or }\mathbb{C}^{n})$ is also open in the usual Euclidean topology of $\mathbb{R}^{n}$ $(\mbox{or }\mathbb{C}^{n}).$ Further, if $\mathcal{O}$ is non-empty, then it is also dense in the usual topology. 
\end{lemma}

\begin{lemma} \label{lem:NonEmZarOpenCandR}
Let $\mathcal{O}$ be a non-empty Zariski open subset of $\mathbb{C}^{n}.$ Then $\mathcal{O} \cap \mathbb{R}^{n}$ is a non-empty Zariski open subset of $\mathbb{R}^{n}.$ 
\end{lemma}

\begin{theorem} \label{thm:Bezout's-Theorem}
For a polynomial system $F(\mathbf{x}) \equiv(f_{1}(\mathbf{x}),\ldots,f_{n}(\mathbf{x})) = 0,$ where $f_{k}:\mathbb{C}^{n} \rightarrow \mathbb{C},$ the total number of its isolated solutions, counting multiplicities, is bounded above by its total degree, i.e., the product $\deg(f_{1})\cdots\deg(f_{n}).$
\end{theorem}

\begin{theorem} \label{thm:ConstFinSolutions}
For $k \in [n],$ let $f_{k}(\mathbf{x};\mathbf{q}):\mathbb{C}^{n} \times \mathbb{C}^{m} \rightarrow \mathbb{C}$ be a polynomial in both $\mathbf{x}$ and $\mathbf{q}.$ Then for the polynomial system $F(\mathbf{x};\mathbf{q})=0,$ where $F(\mathbf{x}; \mathbf{q}) \equiv (f_{1}(\mathbf{x};\mathbf{q}), \ldots, f_{n}(\mathbf{x};\mathbf{q})),$ there exists a non-empty Zariski open set $\mathcal{O} \subset \mathbb{C}^{m}$ such that, for each $\mathbf{q} \in \mathcal{O},$ the system has $r_{i}$ isolated solutions of multiplicity $i,$ where $r_{i}$ is an integer independent of $\mathbf{q}\in\mathcal{O}.$ 
\end{theorem}

\begin{theorem} \label{thm:RegValZarOpen}
Let $F: \mathbb{C}^{d \cdot N_i} \rightarrow \mathbb{C}^{d \cdot N_i}$ be a polynomial map. Then there exists a non-empty Zariski open set $U$ of the co-domain such that for any $\mathbf{u}\in U,$ if $\mathbf{z}$ is a root of $F(\mathbf{x})-\mathbf{u}=0,$ then $\dot{F}(\mathbf{z})$ is non-singular.
\end{theorem}

We are now ready to prove Lemma~\ref{lem:EPSWellBeh}. But we first show its generalization in $\mathbb{C}^{d \cdot N_i}.$  
\begin{lemma} \label{lem:EPSWellBehDup}
There exists an open dense set $\mathcal{C}_i$ of $\mathbb{C}^{d \cdot N_i}$ such that if $\mathbf{w} \in \mathcal{C}_i,$ then the solution set of the EPS  given in \eqref{eq:genElePolSys} satisfies the following properties:

\begin{enumerate}[1.]
\item $|V(H_i)| = k \times N_i!,$ where $k \in \mathbb{Z}_{++}$ is independent of $\mathbf{w} \in \mathcal{C}_i,$ and 

\item Each solution is non-singular.
\end{enumerate}

\end{lemma}
\begin{proof}
Consider the Zariski open set $\mathcal{C}^1 := \{\mathbf{x} \in \mathbb{C}^{d \cdot N_i} : \det(\dot{\mathcal{E}}(\mathbf{x})) \neq 0\}.$ From Lemmas~\ref{lem:egRegPt} and \ref{lem:NonEmpDense}, it is clear that $\mathcal{C}^1$ is non-empty and hence an open dense subset of $\mathbb{C}^{d \cdot N_i}.$ 

For the map $\mathcal{E}$ that is used to define the EPS, let $U \subseteq \mathbb{C}^{d \cdot N_i}$ be the non-empty Zariski open set as guaranteed in Theorem~\ref{thm:RegValZarOpen}. Also, let $\mathcal{C}^2:= \mathcal{E}^{-1}(U).$ Clearly, if $\mathbf{w} \in \mathcal{C}^2,$ then the EPS necessarily has only non-singular roots. Thus, $\mathcal{C}^2 \subseteq \mathcal{C}^1.$ But we now show that $\mathcal{C}^2$ itself is an open dense subset of $\mathbb{C}^{d \cdot N_i}.$ Since $U$ is open and $\mathcal{E}$ is continuous, it follows that $\mathcal{C}^2$ is open. To show that it is dense we assume the contrary. Then clearly, there exists a $\delta > 0$ and $\mathbf{x} \in \mathbb{C}^{d \cdot N_i} \backslash \mathcal{C}^2$ such that $\mathcal{B}(\mathbf{x}; \delta) \cap \mathcal{C}^2 = \emptyset.$ This implies that
\begin{equation} \label{eq:ballUdisjoint}
\mathcal{E}(\mathcal{B}(\mathbf{x}; \delta)) \cap U = \emptyset.
\end{equation}
But $\mathcal{C}^1$ is an open dense subset. Thus, $\mathcal{B}(\mathbf{x}; \delta) \cap \mathcal{C}^1  \neq \emptyset.$ From the inverse function theorem, it follows that, for any $\mathbf{z} \in \mathcal{B}(\mathbf{x}; \delta) \cap \mathcal{C}^1,$ $\mathcal{E}$ is a local homeomorphism at $\mathbf{z}.$ Thus, $\mathcal{E}(\mathcal{B}(\mathbf{x}; \delta))$ has non-empty interior. Equation \eqref{eq:ballUdisjoint} then contradicts the fact that $U$ is an open dense subset. Hence, $\mathcal{C}^2$ is also open dense.

For the map $F(\mathbf{x}, \mathbf{w}): \mathbb{C}^{d \cdot N_i} \rightarrow \mathbb{C}^{d \cdot N_i}$ defined by $F(\mathbf{x}, \mathbf{w}) = \mathcal{E}(\mathbf{x}) - \mathcal{E}(\mathbf{w}),$ let $\mathcal{C}^3$ be the set as guaranteed in Theorem~\ref{thm:ConstFinSolutions}. Then clearly, $\mathcal{C}_i := \mathcal{C}^2 \cap \mathcal{C}^3$ is an open dense subset of $\mathbb{C}^{d \cdot N_i}.$ Furthermore, if $\mathbf{w} \in \mathcal{C}_i,$ then 
\begin{enumerate}[i.]
\item the EPS has only non-singular roots or equivalently $r_i = 0,$ \, $\forall i \neq 1,$ and

\item there are precisely $r_1 \in \mathbb{Z}_{++}$ non-singular roots, where $r_1$ is independent of $\mathbf{w} \in \mathcal{C}_i.$ 
\end{enumerate}
Note that $r_1 > 1$ since $\mathbf{w}$ is always a root. Because of Theorem~\ref{thm:Bezout's-Theorem}, it is also finite. 

From the symmetry of the EPS (see Lemma~\ref{lem:EPSSymm}) and the non-singularity of the roots, it is now easy to see that $|V(H_i)| = k \times N_i!$ for some $k \in \mathbb{Z}_{++}.$ \qquad \end{proof}

Let $\mathcal{C}^1, \mathcal{C}^2, \mathcal{C}^3$ and $U$ be defined as in the proof above.  From Lemmas~\ref{lem:NonEmpDense} and \ref{lem:NonEmZarOpenCandR}, it follows that $\mathcal{R}^1 := \mathcal{C}^1 \cap \mathbb{R}^{d \cdot N_i},$ $U \cap \mathbb{R}^{d \cdot N_i}$ and $\mathcal{R}^3 := \mathcal{C}^3 \cap \mathbb{R}^{d \cdot N_i}$  are open dense subsets of $\mathbb{R}^{d \cdot N_i}.$ The same argument that was used to prove $\mathcal{C}^2$ is an open dense subset of $\mathbb{C}^{d \cdot N_i},$ with the real version of the inversion of the inverse function theorem,  also shows that $\mathcal{R}^2 := \mathcal{E}^{-1}(U \cap \mathbb{R}^{d \cdot N_i}) \cap \mathbb{R}^{d \cdot N_i}$ is an open dense subset of $\mathbb{R}^{d \cdot N_i}.$ But note that all coefficients of $\mathcal{E}$ are real. Hence, $\mathcal{R}^2 = \mathcal{E}^{-1}(U) \cap \mathbb{R}^{d \cdot N_i} = \mathcal{C}^2 \cap \mathbb{R}^{d \cdot N_i}.$ This now implies that $\mathcal{R}_i : = \mathcal{R}^2 \cap \mathcal{R}^3 = \mathcal{C}_i \cap \mathbb{R}^{d \cdot N_i}$ is an open dense subset of $\mathbb{R}^{d \cdot N_i}.$ Since $\mathcal{R}_i \subset \mathcal{C}_i,$ even if $\mathbf{w} \in \mathbb{R}_i$ the EPS has the same properties as those described in Lemma~\ref{lem:EPSWellBehDup}. This now completes the verification Lemma~\ref{lem:EPSWellBeh} as desired. 
\end{appendix}

\section*{Acknowledgments}
The author would like to sincerely thank his advisors, Prof. V.~Borkar and Prof. D.~Manjunath, for inspiring, motivating and guiding him right through this work. He would also like to thank C.~Wampler, A.~Sommese, R.~Gandhi, S.~Gurjar and M.~Gopalkrishnan for helping him understand several results from algebraic geometry. Lastly, he would like to thank S.~Juneja for referring \cite{Botha86}.
 
\bibliographystyle{siam}
\bibliography{tomography}

\end{document}